\documentclass[reqno]{amsart}

\usepackage[utf8]{inputenc}
\usepackage[OT2,T1]{fontenc}
\usepackage[english]{babel}

\usepackage{amsmath}
\usepackage{amsfonts}
\usepackage{amssymb}
\usepackage{amsthm}
\usepackage{mathrsfs}
\usepackage{listings}
\usepackage{hyperref}

\usepackage{multirow}

\usepackage[dvipsnames]{xcolor}
\usepackage{graphicx}
\usepackage[all]{xy}

\usepackage{colortbl}
\definecolor{mygray}{gray}{0.92}

\usepackage{array}
\newcolumntype{C}[1]{>{\centering\arraybackslash$}p{#1}<{$}}

\usepackage{breqn}
\newcounter{myequation}[equation]

\usepackage{algorithm}
\usepackage{algpseudocode}

\numberwithin{equation}{section}

\theoremstyle{plain}
\newtheorem{theorem}[equation]{Theorem}

\newtheorem{proposition}[equation]{Proposition}
\newtheorem{lemma}[equation]{Lemma}

\theoremstyle{definition}
\newtheorem{definition}[equation]{Definition}
\newtheorem{thealgo}[equation]{Algorithm}

\theoremstyle{remark}
\newtheorem{remark}[equation]{Remark}
\newtheorem{example}[equation]{Example}

\def\defi{\emph}
\def\ext{\!\mid\!}

\def\epsilon{\varepsilon}
\def\theta{\vartheta}

\def\Magma{\textsc{Magma}}
\def\SageMath{\textsc{SageMath}}

\def\eg{e.g.}
\def\ie{i.e.}

\def\loccit{loc.\ cit.}

\DeclareMathOperator{\disc}{disc}

\DeclareMathOperator{\DO}{\underline{I}}
\DeclareMathOperator{\SH}{\underline{S}}

\DeclareMathOperator{\End}{End}

\DeclareMathOperator{\Gal}{Gal}
\DeclareMathOperator{\GL}{GL}

\DeclareMathOperator{\im}{Im}

\DeclareMathOperator{\PGL}{PGL}

\DeclareMathOperator{\SL}{SL}

\DeclareMathOperator{\Sp}{Sp}

\def\C{\mathbb{C}}

\def\F{\mathbb{F}}
\def\Fpbar{\overline{\F}_p}
\def\FF{\mathcal{F}}

\def\H{\mathcal{H}}
\def\L{\mathcal{L}}

\def\O{\mathcal{O}}
\def\P{\mathbb{P}\,}
\def\Q{\mathbb{Q}}
\def\R{\mathbb{R}}

\def\Z{\mathbb{Z}}

\def\Qbar{\overline{\Q}}

\usepackage{bm}

\def\Xb{\bm{X}}

\renewcommand{\Im}{\mathrm{Im}}

\hypersetup{
  pdfauthor   = {K{\i}l{\i}\c{c}er, Labrande, Lercier, Ritzenthaler, Sijsling, Streng},
  pdftitle    = {Plane quartics with complex multiplication},
  pdfsubject  = {},
  pdfkeywords = {},
  backref=true, pagebackref=true, hyperindex=true, colorlinks=true,
  breaklinks=true, urlcolor=blue, linkcolor=blue, citecolor=blue,
  bookmarks=true, bookmarksopen=true}

\begin{document}

\title{Plane quartics over $\Q$ with complex multiplication}
\date{\today}

\begin{abstract}
  We give examples of smooth plane quartics over $\Q$ with complex
  multiplication over $\overline{\Q}$ by a maximal order with primitive CM
  type. We describe the required algorithms as we go; these involve the
  reduction of period matrices, the fast computation of Dixmier--Ohno
  invariants, and reconstruction from these invariants. Finally, we discuss
  some of the reduction properties of the curves that we obtain.
\end{abstract}

\author[K{\i}l{\i}\c{c}er]{P{\i}nar K{\i}l{\i}\c{c}er}
\address{%
	P{\i}nar K{\i}l{\i}\c{c}er,
  Carl von Ossietzky Universität Oldenburg,
  Institut für Mathematik,
  26111 Oldenburg, Germany
}
\email{pinarkilicer@gmail.com}

\author[Labrande]{Hugo Labrande}
\address{%
	Hugo Labrande,
  Loria/INRIA Lorraine,
  Équipe CARAMEL,
  615 Rue du jardin botanique,
  54602 Villers-lès-Nancy Cedex,
  France.
}
\address{%
  ISPIA,
  University of Calgary,
  2500 University Dr NW,
  Calgary, Alberta, T2N 1N4,
  Canada.
}
\email{hugo@hlabrande.fr}

\author[Lercier]{Reynald Lercier}
\address{%
  Reynald Lercier,
  DGA \& Univ Rennes, %
  CNRS, IRMAR - UMR 6625, %
  F-35000  Rennes, %
  France. %
}
\email{reynald.lercier@m4x.org}

\author[Ritzenthaler]{Christophe Ritzenthaler}
\address{%
	Christophe Ritzenthaler,
        Univ Rennes, %
        CNRS, IRMAR - UMR 6625, %
        F-35000 Rennes, %
        France. %
}
\email{christophe.ritzenthaler@univ-rennes1.fr}

\author[Sijsling]{Jeroen Sijsling}
\address{
	Jeroen Sijsling,
  Institut f\"ur Reine Mathematik,
  Universit\"at Ulm,
  Helm\-holtzstrasse 18,
  89081 Ulm,
  Germany
}
\email{jeroen.sijsling@uni-ulm.de}

\author[Streng]{Marco Streng}
\address{
	Marco Streng,
  Mathematisch Instituut,
  Universiteit Leiden,
  P.O. box 9512,
  2300 RA Leiden,
  The Netherlands
}
\email{streng@math.leidenuniv.nl}

\thanks{The third and fourth author acknowledge support from the CysMoLog
``d\'efi scientifique \'emergent'' of the Universit\'e de Rennes 1.
The sixth author was supported by the Netherlands Organisation for Scientific
Research (NWO) Vernieuwingsimpuls Veni.
The authors would like to thank Enea Milio and the anonymous referee
for various comments that were helpful for the improvement of the exposition.}


\maketitle

\section*{Introduction}

Abelian varieties with complex multiplication (CM) are a fascinating common
ground between algebraic geometry and number theory, and accordingly have been
studied since a long time ago. One of the highlights of their theoretical study
was the proof of Kronecker's \emph{Jugendtraum}, which describes the ray class
groups of imaginary quadratic fields in terms of the division points of
elliptic curves. Hilbert's twelfth problem asked for the generalization of this
theorem to arbitrary number fields, and while the general version of this
question is still open, Shimura and Taniyama~\cite{shimura61:_compl_abelian}
gave an extensive partial answer for CM fields by using abelian varieties whose
endomorphism algebras are isomorphic to these fields. A current concrete
application of the theory of CM abelian varieties is in public key
cryptography, where one typically uses this theory to construct elliptic curves
with a given number of points~\cite{broeker-stevenhagen-cm}.

Beyond the theoretically well-understood case of elliptic curves, there are
constructions of curves with CM Jacobians in both genus $2$~\cite{spallek,
wamelen, BouyerStreng} and $3$~\cite{kowe, wenghypg3, LarioSomoza}. Note that
in genus~$2$ every curve is hyperelliptic, which leads to a relatively simple
moduli space; moreover, the examples in genus $3$ that we know up to now are
either hyperelliptic or Picard curves, which again simplifies considerations.
This paper gives the first $19$ conjectural examples of ``generic'' CM curves
of genus~$3$, in the sense that the curves obtained are smooth plane quartics
with trivial automorphism group. More precisely, it conjecturally completes the
list of curves of genus~$3$ over $\Q$ whose endomorphism rings over $\Qbar$ are
maximal orders of sextic fields (see Theorem \ref{thm:list_of_fields}). The
other curves of genus~$3$ with such endomorphism rings are either hyperelliptic
or Picard curves. The hyperelliptic ones were known to Weng~\cite{wenghypg3},
except for three curves that were computed by Balakrishnan, Ionica,
K{\i}l{\i}\c{c}er, Lauter, Vincent, Somoza and Streng by using the methods and
\SageMath{} implementation of~\cite{BILV,BILVcode}. The Picard curves had all
previously appeared in work by Koike-Weng~\cite{kowe} and
Lario-Somoza~\cite{LarioSomoza}.

To construct our curves, we essentially follow the classical path; first we
determine the period matrices, then the corresponding invariants, then we
reconstruct the curves from rational approximations of these invariants, and
finally we heuristically check that the curves obtained indeed have CM by the
correct order. In genus~$3$, however, all of these steps are somewhat more
complicated than was classically the case.

The proven verification that the curves obtained indeed have CM by the correct
order is left for another occasion; we restrict ourselves to a few remarks.
First of all, there are no known equivalents in genus~$3$ of the results that
bound the denominators of Igusa class polynomials~\cite{viray}. In fact very
little is known on the arithmetic nature of the Shioda and Dixmier--Ohno
invariants that are used in genus~$3$, and a theoretical motivation for finding
our list was to have concrete examples to aid with the generalization of the
results in \loccit{}

Using the methods in~\cite{cmsv-article} one could still verify the
endomorphism rings of our curves directly; this has already been done for the
simplest of our curves, namely
\begin{multline*} X_{15} : {{x}}^{4} -
  {{x}}^{3}{y} + 2\,{{x}}^{3}{z} + 2\,{{x}}^{2}{y}\,{z} + 2\,{{x}}^{2}{{z}}^{2}
  - 2\,{x}\,{{y}}^{2}{z} +\\ 4\,{x}\,{y}\,{{z}}^{2} - {{y}}^{3}{z} +
  3\,{{y}}^{2}{{z}}^{2} + 2\,{y}\,{{z}}^{3} + {{z}}^{4} = 0\,.
\end{multline*}
The main restriction for applying these methods to the other examples is the
time required for this verification. At any rate, the results in the final
section of this paper are coherent with the existence of a CM structure with
the given order.

The CM fields that give rise to our curves were determined by arithmetic
methods in~\cite{kilicerthesis, KS2016}. This also gives us Riemann matrices
that we can use to determine periods and hence the invariants of our quartic
curves. However, we do need to take care to reduce our matrices in order to get
good convergence properties for their theta values. The theory and techniques
involved are discussed in Section~\ref{sec:periods}.

With our reduced Riemann matrices in hand, we want to calculate the
corresponding theta values. We will need these values to high precision so as
to later recognize the corresponding invariants. The fast algorithms needed to
make this feasible were first developed in~\cite{Labrande} for genus~$2$;
further improvements are discussed in Section~\ref{sec:fastevaluation}. In the
subsequent Section~\ref{sec:do} we indicate how these values allow us to obtain
the Dixmier--Ohno invariants of our smooth plane quartic curves. This is based
on formulas obtained by Weber~\cite{weber, fiorentino}.

The theory of reconstructing smooth plane quartics from their invariants was
developed in~\cite{LRS16} and is a main theme of Section~\ref{sec:rec}. Equally
important is the performance of these algorithms, which was substantially
improved during the writing of this paper; starting from a reasonable tuple of
Dixmier--Ohno invariants over $\Q$, we now actually obtain corresponding plane
quartics over $\Q$ with acceptable coefficients, which was not always the case
before. In particular, we developed a ``conic trick'' which enables us to find
conics with small discriminant in the course of Mestre's reconstruction
algorithms for general hyperelliptic curves (by \loccit, the reconstruction
methods for non-hyperelliptic curves of genus $3$ reduce to Mestre's algorithms
for the hyperelliptic case). Section~\ref{sec:rec} discusses these and other
speed-ups and the mathematical background from which they sprang. Without them,
our final equations would have been too large to even write down.

We finally take a step back in Section~\ref{sec:remarks} to examine the
reduction properties of these curves, as well as directions for future work,
before giving our explicit list of curves in Section~\ref{sec:results}.

\section{Riemann matrices}\label{sec:periods}

Let $A$ be a principally polarized abelian variety of dimension $g$ over~$\C$,
such as the Jacobian $A=J(C)$ of one of the curves that we are looking for.
Then by integrating over a symplectic basis of the homology and normalizing,
the manifold $A$ gives rise to a point $\tau$ in the Siegel upper half space
$\H_g$, well-defined up to the action of the symplectic group $\Sp_{2 g} (\Z)$.
The elements of $\H_g$ are also known as \defi{Riemann matrices}. In
Section~\ref{sec:listoffields}, we give the list, due to K\i{}l\i{}\c{c}er and
Streng, of all fields $K$ that can occur as endomorphism algebra of a simple
abelian threefold over $\Q$ with complex multiplication over~$\Qbar$. In
Section~\ref{sec:fieldtomatrix}, we recall Van Wamelen's methods for listing
all Riemann matrices with complex multiplication by the maximal order of a
given field. In Section~\ref{sec:reducematrix}, we show how to reduce Riemann
matrices to get Riemann matrices with better convergence properties.

\subsection{The CM fields}
\label{sec:listoffields}

Let $A$ be an abelian variety of dimension $g$ over a field $k$ of
characteristic~$0$, let $K$ be a number field of degree $2g$ and let
$\mathcal{O}$ be an order in~$K$. We say that $A$ \defi{has CM} by
$\mathcal{O}$ (over $\overline{k}$) if there exists an embedding $\mathcal{O}
\rightarrow \End(A_{\overline{k}})$.

If $A$ is simple over $\overline{k}$ and has CM by the full ring of integers
$\mathcal{O}_K$ of $K$, then we have in fact $\mathcal{O}_K\cong
\End(A_{\overline{k}})$ and $K$ is a \defi{CM field}, \ie{}, a totally
imaginary quadratic extension $K$ of a totally real number field
$F$~\cite{LangCM}.

The \defi{field of moduli} of a principally polarized abelian variety $A/k$ is
the residue field of the corresponding point in the moduli space of principally
polarized abelian varieties. It is also the intersection of the fields of
definition of $A$ in $\overline{k}$ \cite[p.37]{koizumi}. In particular, if $A$
is defined over $\Q$, then its field of moduli is~$\Q$. The field of moduli  of
a curve or an abelian variety is not always a field of
definition~\cite{shimura-moduli}. However, we have the following theorem.

\begin{theorem}  \label{thm:list_of_fields}
  There are exactly $37$ isomorphism classes of CM fields $K$ for which there
  exist principally polarized abelian threefolds $A/\Qbar$ with field of
  moduli~$\Q$ and $\End(A)\cong \mathcal{O}_K$. The set of such fields is
  exactly the list of fields given in Table~\ref{table: CMfields}.

  For each such field $K$, there is exactly one such principally polarized
  abelian variety $A$ up to $\Qbar$-isomorphism, and this variety is the
  Jacobian of a curve $X$ of genus $3$ defined over~$\Q$. In particular, the
  abelian variety $A$ itself is defined over $\Q$.
\end{theorem}

\begin{proof}
  The first part, up to and including uniqueness of $A$, is exactly
  Theorem~4.1.1 of K\i{}l\i{}\c{c}er's thesis~\cite{kilicerthesis}. These $37$
  cases are listed in Table~\ref{table: CMfields}. Therefore we need only prove
  the statement on the field of definition, which can be done here directly
  from the knowledge of the CM field. By the theorem of
  Torelli~\cite[Appendix]{lauter}, $k$ is a field of definition for the
  principally polarized abelian threefold $A$ if and only if it is a field of
  definition for $X$. This implies that the field of moduli of $X$ equals $\Q$,
  and we have to show that this field of moduli is also a field of definition.

  In genus $3$ all curves descend to their field of moduli, except for plane
  quartics with automorphism group $\Z / 2 \Z$ and hyperelliptic curves with
  automorphism group $\Z / 2 \Z \times \Z / 2 \Z$ (see \cite{LRRS14,
  LRS-Desc}). We finish by showing that neither of these occurs in
  Table~\ref{table: CMfields}. If $\Q(i)$ is a subfield of $K$, then by Weng
  \cite[\S 4.4--4.5]{wenghypg3}, the curve $X$ is hyperelliptic with
  automorphism group containing $\Z / 4\Z$, in which case it descends to its
  field of moduli. We therefore assume the contrary. If the curve $X$ over
  $\Qbar$ is hyperelliptic, then its automorphism group is the group $\mu_K$
  itself. Since this group is cyclic, it cannot be isomorphic to $\neq \Z / 2\Z
  \times \Z / 2 \Z$ and the curve $X$ descends to its field of moduli. If $X$
  is non-hyperelliptic, then its automorphism group is $\mu_K/\{\pm 1\}$.
  Because of our assumption on $K$, this group is not isomorphic to $\Z/2\Z$,
  and again $X$ descends to $\Q$.
\end{proof}

Table~\ref{table: CMfields} gives a list of cyclic sextic CM fields~$K$,
arranged as follows. Let $K$ be such a field. Then it has an imaginary
quadratic subfield $k$ and a totally real cubic subfield~$F$. In
Table~\ref{table: CMfields}, the number $d_k$ is the discriminant of~$k$; the
polynomial $p_F$ is a defining polynomial for $F$. These two entries of the
table define the field~$K$. The number $f_F$ is the conductor of $F$, and $d_K$
is the discriminant of $K$. The entry $\#$ is the order of the automorphism
group of the Jacobian of the corresponding curve, which is nothing but the
number of roots of unity in~$K$. The ``Type'' column indicates whether the
conjectured model of the curve is hyperelliptic (H), Picard (P), or a plane
quartic with trivial automorphism group~(G). The ``Curve'' column gives a
reference to the conjectured model over $\Q$ of the curve. The cases $1$, $2$,
$3$, $5$, \ldots, $20$ correspond to the smooth plane quartics $X_i$ in
Section~\ref{sec:results}.

\newcommand{\refbiklssvplus}{\cite{BILVcode, LR11}+see \S\ref{sec:listoffields}}

\begin{table}[!htbp]
\centering
\resizebox{12.5cm}{!}{
\begin{tabular}{| c | c  | l | c | c | c | c | c | c |} \hline
Case    &$-d_k$& \hspace{12mm}  $p_F$    & $f_F$        & $-d_K$          &$h_K^*$&$\#$ & Type & Curve\\ \hline
  1     & $7$ & $X^3 + X^2 - 4X + 1$    & $13$         & $7^3\cdot13^4$   &$1$  & $2$ & G & $X_{1}$ (§\,\ref{sec:results})\\ \hline
  2     & $7$ & $X^3 - 3X - 1$          & $3^2$        & $3^8\cdot7^3$    &$1$  & $2$ & G & $X_{2}$ (§\,\ref{sec:results})\\ \hline
  3     & $7$ & $X^3 + 8X^2 - 51X + 27$ & $7\cdot 31$  & $7^5\cdot31^4$   &$1$  & $2$ & G & $X_{3}$ (§\,\ref{sec:results})\\ \hline
  4     & $7$ & $X^3 + 6X^2 - 9X + 1$   & $3^2\cdot 7$ & $3^8\cdot7^5$    &$1$  & $2$ & H &
\refbiklssvplus
\\ \hline
  5     & $7$ & $X^3 + X^2 - 30X + 27$  & $7\cdot 13$  & $7^5\cdot13^4$   &$1$  & $2$ & G & $X_{5}$ (§\,\ref{sec:results}) \\ \hline
  6     & $7$ & $X^3 + 4X^2 - 39X + 27$ & $7\cdot 19$  & $7^5\cdot19^4$   &$1$  & $2$ & G & $X_{6}$ (§\,\ref{sec:results}) \\ \hline
  7     & $7$ & $X^3 + X^2 - 24X - 27$  & $73$         & $7^3\cdot73^4$   &$4$  & $2$ & G & $X_{7}$ (§\,\ref{sec:results}) \\ \hline
  8     & $7$ & $X^3 + 2X^2 - 5X + 1$   & $19$         & $7^3\cdot19^4$   &$4$  & $2$ & G & $X_{8}$ (§\,\ref{sec:results}) \\ \hline
  9     & $8$ & $X^3 + X^2 - 4X + 1$    & $13$         & $2^9\cdot13^4$   &$1$  & $2$ & G & $X_{9}$ (§\,\ref{sec:results}) \\ \hline
  10    & $8$ & $X^3 + X^2 - 2X - 1$    & $7$          & $2^9\cdot7^4$    &$1$  & $2$ & G & $X_{10}$ (§\,\ref{sec:results}) \\ \hline
  11    & $8$ & $X^3 + X^2 - 10X - 8$   & $31$         & $2^9\cdot31^4$   &$4$  & $2$ & G & $X_{11}$ (§\,\ref{sec:results}) \\ \hline
  12    & $11$& $X^3 + X^2 - 2X - 1$    & $7$          & $7^4\cdot11^3$   &$1$  & $2$ & G & $X_{12}$ (§\,\ref{sec:results}) \\ \hline
  13    & $11$& $X^3 + X^2 - 14X + 8$   & $43$         & $11^3\cdot43^4$  &$4$  & $2$ & G & $X_{13}$ (§\,\ref{sec:results}) \\ \hline
  14    & $11$& $X^3 + 2X^2 - 5X + 1$   & $19$         & $11^3\cdot19^4$  &$4$  & $2$ & G & $X_{14}$ (§\,\ref{sec:results}) \\ \hline
  15    & $19$& $X^3 + 2X^2 - 5X + 1$   & $19$         & $19^5$           &$1$  & $2$ & G & $X_{15}$ (§\,\ref{sec:results}) \\ \hline
  16    & $19$& $X^3 - 3X - 1$          & $3^2$        & $3^8\cdot19^3$   &$4$  & $2$ & G & $X_{16}$ (§\,\ref{sec:results}) \\ \hline
  17    & $19$& $X^3 + 9X^2 - 30X + 8$  & $3^2\cdot 19$& $3^8\cdot19^5$   &$1$  & $2$ & G & $X_{17}$ (§\,\ref{sec:results}) \\ \hline
  18    & $19$& $X^3 + 7X^2 - 66X - 216$& $13\cdot 19$ & $13^4\cdot19^5$  &$1$  & $2$ & G & $X_{18}$ (§\,\ref{sec:results}) \\ \hline
  19    & $43$& $X^3 + X^2 - 14X+ 8$    & $43$         & $43^5$           &$1$  & $2$ & G & $X_{19}$ (§\,\ref{sec:results}) \\ \hline
  20    & $67$& $X^3 + 2X^2 - 21X - 27$ & $67$         & $67^5$           &$1$  & $2$ & G & $X_{20}$ (§\,\ref{sec:results}) \\ \hline
  21    & $4$ & $X^3 + 2X^2 - 5X + 1$   & $19$         & $2^6\cdot19^4$   &$1$  & $4$ & H &~\cite[\S6 3rd ex.]{wenghypg3} \\ \hline
  22    & $4$ & $X^3 - 3X - 1$          & $3^2$        & $2^6\cdot3^8$    &$1$  & $4$ & H &~\cite[\S6 2nd ex.]{wenghypg3} \\ \hline
  23    & $4$ & $X^3 + X^2 - 2X - 1$    & $7$          & $2^6\cdot7^4$    &$1$  & $4$ & H &\cite{TT91} (also~\cite[\S6 1st ex.]{wenghypg3}) \\ \hline
  24    & $4$ & $X^3 + X^2 - 10X - 8$   & $31$         & $2^6\cdot31^4$   &$4$  & $4$ & H &~\cite[\S6 4th ex.]{wenghypg3} \\ \hline
  25    & $4$ & $X^3 + X^2 - 14X + 8$   & $43$         & $2^6\cdot43^4$   &$4$  & $4$ & H &\refbiklssvplus \\ \hline
  26    & $4$ & $X^3 + 3X^2 - 18X + 8$  & $3^2\cdot 7$ & $2^6\cdot3^8 7^4$&$4$  & $4$ & H &\refbiklssvplus \\ \hline
  27    & $3$ & $X^3 + X^2 - 4X + 1$    & $13$         & $3^3\cdot13^4$   &$1$  & $6$ & P &~\cite[6.1(3)]{kowe} (also~\cite[4.1.3]{LarioSomoza})\\ \hline
  28    & $3$ & $X^3 + X^2 - 2X- 1$     & $7$          & $3^3\cdot7^4$    &$1$  & $6$ & P &~\cite[6.1(2)]{kowe} (also~\cite[4.1.2]{LarioSomoza}) \\ \hline
  29    & $3$ & $X^3 + X^2 - 10X - 8$   & $31$         & $3^3\cdot31^4$   &$1$  & $6$ & P & ~\cite[6.1(4)]{kowe} (also~\cite[4.1.4]{LarioSomoza})  \\ \hline
  30    & $3$ & $X^3 + X^2 - 14X+ 8$    & $43$         & $3^3\cdot43^4$   &$1$  & $6$ & P &~\cite[6.1(5)]{kowe} (also~\cite[4.1.5]{LarioSomoza})  \\ \hline
  31    & $3$ & $X^3 + 3X^2 - 18X + 8$  & $3^2\cdot 7$ & $3^9\cdot7^4$    &$1$  & $6$ & P &~\cite[4.2.1.1]{LarioSomoza} \\ \hline
  32    & $3$ & $X^3 + 6X^2 - 9X + 1$   & $3^2\cdot 7$ & $3^9\cdot7^4$    &$1$  & $6$ & P &~\cite[4.2.1.2]{LarioSomoza} \\ \hline
  33    & $3$ & $X^3 + 3X^2 - 36X - 64$ & $3^2\cdot 13$& $3^9\cdot13^4$   &$1$  & $6$ & P &~\cite[4.2.1.3]{LarioSomoza} \\ \hline
  34    & $3$ & $X^3 + 4X^2 - 15X - 27$ & $61$         & $3^3\cdot61^4$   &$4$  & $6$ & P &~\cite[4.3.1]{LarioSomoza} \\ \hline
  35    & $3$ & $X^3 + 2X^2 - 21X - 27$ & $67$         & $3^3\cdot67^4$   &$4$  & $6$ & P &~\cite[4.3.2]{LarioSomoza} \\ \hline
  36    & $7$ & $X^3 + X^2 - 2X - 1$    & $7$          & $7^5$            &$1$  &$14$ & H & $y^2 = x^7-1$
  \\ \hline
  37    & $3$ & $X^3 - 3X - 1$          & $3^2$        & $3^9$            &$1$  & $18$& P &  $y^3 = x^4-x$
\\ \hline
 \end{tabular}
}
\vspace{0,5cm}
\caption{\label{table: CMfields}
CM fields in genus $3$ whose maximal orders give rise to CM curves
with field of moduli~$\Q$, sorted by the order $\#$ of the group of roots of
unity.
}
\end{table}

In the hyperelliptic cases, curves can be reconstructed by applying the
\SageMath{}~\cite{sage} code of Balakrishnan, Ionica, Lauter and
Vincent~\cite{BILVcode} (based on~\cite{wenghypg3, BILV}) and
\Magma{}~\cite{Magma} functionality due to Lercier and Ritzenthaler for
hyperelliptic reconstruction in genus $3$~\cite{LR11}.

Some of these curves were already computed by Weng~\cite{wenghypg3}. The final
cases $4$, $25$, $26$ were found by Balakrishnan, Ionica, K{\i}l{\i}\c{c}er,
Lauter, Somoza, Streng and Vincent and will appear online soon. The Picard
curves can be obtained as a special case of our construction, but are more
efficiently obtained using the methods of Koike--Weng~\cite{kowe} and
Lario--Somoza~\cite{LarioSomoza}. The rational models in~\cite{wenghypg3, kowe,
LarioSomoza} as well as those that can be obtained with~\cite{BILVcode,
LR11} are correct up to some precision over~$\C$. In case~23, the hyperelliptic
model was proved to be correct in Tautz--Top--Verberkmoes~\cite[Proposition
4]{TT91}. The hyperelliptic model $y^2 = x^7-1$ for case 36 is a classical
result (see Example (II) on page 76 in Shimura~\cite{Shimura77}) and the Picard
model $y^3 = x^4-x$ for case 37 is similar (\eg{}\
Bouw--Cooley--Lauter--Lorenzo--Manes--Newton--Ozman~\cite[Lemma 5.1]{BCLLMNO});
both can be proven by exploiting the large automorphism group of the curve.

\begin{remark}
  In fact the curve in Case $4$ also admits a hyperelliptic defining equation
  over~$\Q$, which is not automatic; \textit{a priori} it is a degree $2$ cover
  of conic that we do not know to be isomorphic to $\P^1$. However, in this
  case the algorithms in~\cite{cmsv-article} show that the conjectural model
  obtained is correct, so that also in this case a hyperelliptic model exists
  over the field of moduli~$\Q$.
\end{remark}

In this paper, we construct models for the generic plane quartic cases.

\subsection{Obtaining Riemann matrices from CM fields}
\label{sec:fieldtomatrix}

Let $\L$ be a lattice of full rank $2g$ in a complex $g$-dimensional vector
space~$V$. The quotient $V/\L$ is a complex Lie group, called a \defi{complex
torus}. This complex manifold is an abelian variety if and only if it is
projective, which is true if and only if there exists a \defi{Riemann form}
for~$\L$, that is, an $\R$-bilinear form $E : V\times V \longrightarrow \R$
such that $E(\L, \L)\subset \Z$ and such that the form
\begin{align}
  V\times V &\longrightarrow \R\\
  (u,v) &\longmapsto E(u, iv) \nonumber
\end{align}
is symmetric and positive definite. The Riemann form is called a
\defi{principal polarization} if and only if the form $E$ on $\L$ has
determinant equal to~$1$. We call a basis $(\lambda_1,\ldots, \lambda_{2g})$ of
$\L$ \defi{symplectic} if the matrix of $E$ with respect to the basis is given
in terms of $g\times g$ blocks as
\begin{equation}
  \begin{split}
    \Omega_g =
    \left( \begin{matrix}
    0 & I_g \\
    -I_g & 0
    \end{matrix} \right).
  \end{split}
\end{equation}
For every principal polarization, there exists a symplectic basis. If we write
out the elements of a symplectic basis as column vectors in terms of a
$\C$-basis of $V$, then we get a $g\times 2g$ \defi{period matrix}.

The final $g$ elements of a symplectic basis of $\L$ for $E$ form a $\C$-basis
of $V$, so we use this as our basis of~$V$. Then the period matrix takes the
form $(\tau \ |\ I_{g})$, where the $g\times g$ complex matrix~$\tau$ has the
following properties:
\begin{itemize}
  \item[(1)] $\tau$ is symmetric,
  \item[(2)] $\im(\tau)$ is positive definite.
\end{itemize}
The set of such matrices forms the Siegel upper half space $\H_g$. Conversely,
from every Riemann matrix~$\tau$, we get the complex abelian variety
\begin{equation}
  \C^g/(\tau \Z^g + \Z^g)
\end{equation}
which we can equip with a Riemann form given by $\Omega_g$ with respect to the
basis given by the columns of $(\tau \ |\ I_{g})$.

Given a CM field~$K$, Algorithm~1 of Van Wamelen~\cite{wamelen} (based on the
theory of Shimura--Taniyama~\cite{shimura61:_compl_abelian}) computes at least
one Riemann matrix for each isomorphism class of principally polarized abelian
variety with CM by the maximal order of~$K$. For details, and an improvement
which computes \emph{exactly} one Riemann matrix for each isomorphism class,
see also Streng~\cite{streng, strengrecip}. In our implementation, we could simplify the
algorithm slightly, because the group appearing in Step~2
of~\cite[Algorithm~1]{wamelen} is computed by
K\i{}l\i{}\c{c}er~\cite[Lemma~4.3.4]{kilicerthesis} for the fields in
Table~\ref{table: CMfields}.

\subsection{Reduction of Riemann matrices}
\label{sec:reducematrix}

There is an action on the Siegel upper half space $\H_g$ by the symplectic
group
\begin{equation}
  \Sp_{2g}(\Z)
  =
  \{ M\in\GL_{2g}(\Z): M^t\Omega_g M
  =
  \Omega_g\},
\end{equation}
given by
\begin{equation}
  \left( \begin{matrix}
    A & B \\
    C & D
  \end{matrix} \right) (\tau)
  =
  (A\tau+B)(C\tau+D)^{-1}.
\end{equation}

The isomorphism class of principally polarized abelian variety
$(\C^g/(\tau\Z^g+\Z^g), \Omega_g)$ of Section~\ref{sec:fieldtomatrix} depends
only on the orbit of $\tau$ under the action of $\Sp_{2g}(\Z)$, so we change
$\tau$ into an $\Sp_{2g}(\Z)$-equivalent matrix on which the theta constants
have faster convergence. For this, we use \cite[Algorithm 2 in
\S4.1]{labrande-thome}. To avoid numerical instability, we replace the
condition $|\tau_{11}^{\prime}|\leq 1$ in Step~3 of loc.\ cit.\ by
$|\tau_{11}^{\prime}|< 0.99$. The result of this reduction then is a matrix
$\tau\in\H_g$ such that the real parts of all entries have absolute value $\leq
\frac{1}{2}$, such that the upper left entry has absolute value
$|\tau_{11}|\geq 0.99$ and such that the imaginary part $Y = \Im(\tau)$ is
\defi{Minkowski-reduced}, \ie,
\begin{itemize}
\item[(a)] for all $j =1, \ldots, g$ and all $v=(v_1,\ldots,v_g)\in\Z^g$ with
  $\gcd(v_j,\ldots, v_g)=1$, we have ${}^tv Y v \geq Y_{j\,j}$, and
\item[(b)] for all $j = 1,\ldots,g-1$, we have $Y_{j\,j+1}\geq 0$.
\end{itemize}
For example, taking $i\not=j$ and $v=e_i\pm e_j$ in (a) gives $Y_{ii}\pm
2Y_{ij}\geq 0$, while taking $j=1$ and $v=e_i$ in (a) gives $Y_{ii}\geq
Y_{11}$, so
\begin{align}\label{eq:minkowskiexample}
  |Y_{ij}|\leq \frac{1}{2}Y_{ii},\quad Y_{11}\leq Y_{ii}.
\end{align}
We also have
\begin{align}
  \Im(\tau_{11}) \geq \sqrt{0.99^2-0.5^2} > 0.85.\label{eq:85}
\end{align}

We have implemented this algorithm, as well as a version with Minkowski reduction
replaced by LLL reduction, which scales better with $g$ than Min\-kowski
reduction. For a self-contained exposition including a proof that the
LLL-version of the algorithm terminates, see the first arXiv version\footnote{
  \href{https://arxiv.org/abs/1701.06489v1}{https://arxiv.org/abs/1701.06489v1}
} of our paper, and for an alternative approach with LLL reduction, see
Deconinck, Heil, Bobenko, van Hoeij, and Schmies~\cite{DHBHS}.

\subsection{Conclusion and efficiency}

It takes only a minute to compute the reduction with either the Minkowski or
the LLL version of the reduction algorithm for all our Riemann matrices. We did
not notice any difference in efficiency of numerical evaluations of
Dixmier-Ohno invariants (as in Section~\ref{sec:computingDO}) between the
Minkowski-version of the reduction and the LLL-version of the reduction.
Without reduction, we were unable to do the computations to sufficient
precision for reconstructing all the curves. We conclude that for $g=3$, there
is no reason to prefer one of these algorithms over the other, but it is very
important to use at least one of them. We do advise caution with the
LLL-version, as the analysis in Section~\ref{sec:computingDO} below is valid
only for Minkowski-reduced matrices.

\section{Computing the Dixmier--Ohno invariants}
\label{sec:computingDO}

In this section, we show how given a Riemann matrix $\tau$ we can obtain an
approximation of the Dixmier--Ohno invariants of a corresponding plane quartic
curve. One procedure has been described in~\cite{guardia} and relies on the
computation of derivatives of odd theta functions. Here we take advantage of
the existence of fast strategies to compute the Thetanullwerte to emulate the
usual strategy for such computations in the hyperelliptic case~\cite{wenghypg3,
BILV}: we use an analogue of the Rosenhain formula to compute a special
\defi{Riemann model} for the curve from the Thetanullwerte, from which we then
calculate an approximation of the Dixmier--Ohno invariants. By normalizing
these, we find an explicit conjectural representative of the Dixmier--Ohno
invariants as an element of a weighted projective space over $\Q$.

\subsection{Fast computation of the Thetanullwerte from a Riemann matrix}
\label{sec:fastevaluation}

\begin{definition}\label{def:theta}
  The \defi{Thetanullwerte} or \defi{theta-constants} of a Riemann matrix
  $\tau\in \H_3$ are defined as
  \begin{equation}
    \theta_{[a;b]}(0, \tau)
    =
    \sum_{n \in \mathbb{Z}^3}
    e^{i \pi \left({}^t\left(n+a\right) \tau (n+a) + 2{}^t (n+a)b\right)},
  \end{equation}
  where $a,b \in \{0,{1}/{2}\}^3$. We define the \defi{fundamental
  Thetanullwerte} to be those $\theta_{[a;b]}$ with $a=0$; there are $8$ of
  them.
\end{definition}

In many applications, only the $36$ so-called \defi{even} Thetanullwerte are
considered, which are those for which the dot product $4a\cdot b$ is even. The
other Thetanullwerte turn out to always be equal to~$0$.

We further simplify notation by writing
\begin{equation}
  \theta_{[a;b]} = \theta_i, \qquad i
  = 2(b_0 + 2b_1 + 4b_{2})
  + 2^{4}(a_0 + 2 a_1 + 4a_2).
\end{equation}
In other words, we number the Thetanullwerte by interpreting the reverse of the
sequence $(2b||2a)$ as a binary expansion. This is the numbering used in,
\eg{},~\cite{Dupont,Labrande}. For notational convenience, we write
$\theta_{n_1, \ldots, n_k}$ for the $k$-tuple $\theta_{n_1}, \ldots,
\theta_{n_k}$. In this section, we describe a fast algorithm to
compute the Thetanullwerte with high precision. Note that it is sufficient to
describe an algorithm that computes the fundamental Thetanullwerte; we can then
compute the squares of all $64$ Thetanullwerte by computing the fundamental
ones at ${\tau}/{2}$, then use the following $\tau$-duplication
formula~\cite[Chap.~IV]{igusa1}:
\begin{equation}
  \theta_{[a;b]}(0,\tau)^2 =
  \frac{1}{2^3} \sum_{\beta \in \frac{1}{2} \mathbb{Z}^3 / \mathbb{Z}^3}
  e^{-4i\pi {}^ta\beta} \theta_{[0;b+\beta]} \left(0, \frac{\tau}{2}\right)
  \theta_{[0;\beta]} \left(0,\frac{\tau}{2}\right).
\end{equation}
We can then recover the $64$ Thetanullwerte from their squares, by using a
low-precision approximation of their value to decide on the appropriate square
root. Both algorithms described in this subsection have been implemented in
\Magma~\cite{ourimplementation}.

\subsubsection{Naive algorithm for the Thetanullwerte}\label{sec:naive}

A (somewhat) naive algorithm to compute the Thetanullwerte consists in
computing the sum in Definition~\ref{def:theta} until the remainder is too
small to make a difference at the required precision. We show in this section
that it is possible to compute the genus 3 Thetanullwerte up to precision $P$
(that is, up to absolute difference of absolute value at most $10^{-P}$) by
using $O(\mathcal{M}(P) P^{1.5})$ bit operations. Here $\mathcal{M}(P)$ is the
number of bit operations needed for one multiplication of $P$-bit integers.
This running time is the same as for the general strategy given
in~\cite{Deconinck}, as analyzed in~\cite[Section~5.3]{Labrande}.

Let $t_{m,n,p} = e^{i \pi (m,n,p) \tau {}^t(m,n,p)}$, so $\theta_0(\tau) =
\sum_{m, n, p \in\Z} t_{m,n,p}$. Our algorithm computes the approximation
\begin{equation}
  S_B = \sum_{m, n, p \in [-B, B]} t_{m,n,p}
\end{equation}
of $\theta_0(\tau)$. The main idea is to use the following recurrence relation.
Let $q_{jk} = e^{i \pi \tau_{jk}}$. Then we have
\begin{align}
  t_{m+1,n,p}^{\vphantom{p}} & = t_{m,n,p}^{\vphantom{p}} q_{11}^{2m\vphantom{p}} q_{11}^{\vphantom{p}} q_{12}^{2n\vphantom{p}} q_{13}^{2p}, \nonumber\\
  t_{m,n+1,p}^{\vphantom{p}} & = t_{m,n,p}^{\vphantom{p}} q_{22}^{2n\vphantom{p}} q_{22}^{\vphantom{p}} q_{12}^{2m\vphantom{p}} q_{23}^{2p}, \\
  t_{m,n,p+1}^{\vphantom{p}} & = t_{m,n,p}^{\vphantom{p}} q_{33}^{2p} q_{33}^{\vphantom{p}} q_{23}^{2n\vphantom{p}} q_{13}^{2m\vphantom{p}}.\nonumber
\end{align}

\begin{thealgo}[Given a period matrix $\tau \in \FF_3(\{N\})$ and a bound $B$,
  compute $S_B$.]\label{alg:thetanullwerteNaive}{}\
  \begin{enumerate}
    \item $S_B \leftarrow t_{0,0,0} = 1$.
    \item For $m = 1, 2, \ldots, B$ and $m = -1, -2, \ldots, -B$:
    \item \qquad Compute $t_{m,0,0}$ using the recursion and add it to $S_B$.
    \item \qquad For $n = 1, 2, \ldots, B$ and $n = -1, -2, \ldots, -B$:
    \item \qquad \qquad Compute $t_{m,n,0}$ using the recursion and add it to
      $S_B$.
    \item \qquad \qquad For $p = 1, 2, \ldots, B$ and $p = -1, -2, \ldots, -B$:
    \item \qquad \qquad \qquad Compute $t_{m,n,p}$ using recursion and add it
      to $S_B$.
    \item Return $S_B$.
  \end{enumerate}
\end{thealgo}

This algorithm can be modified to compute approximations of any fundamental
Thetanullwerte $\theta_{[0,b]}$ by adjusting the sign of each term (with a
factor $(-1)^{(m,n,p).b}$). Hence, the computation of $S_B$ reduces to the
computation of the $q_i$ and the use of the recursion relations to compute each
term. We prove in the rest of the section that, for this algorithm to compute
$\theta_0$ up to $2^{-P}$, taking $B = O(\sqrt{P})$ is sufficient. That is, we
prove that
\begin{align}\label{eq:precisionnaive}
  |\theta_0(\tau) - S_B| < 2^{-P} \qquad
  \mbox{for an easily computable $B = O(\sqrt{P})$}.
\end{align}
This allows the computation of the genus 3 Thetanullwerte in $O(\mathcal{M}(P)
P^{1.5})$; we refer to our implementation~\cite{ourimplementation} of the naive
algorithm for full details.

Our analysis is similar to the ones in~\cite{Dupont,Labrande}. We use the
following lemma, of which we defer the proof until the end of
\S\ref{sec:naive}.
\begin{lemma}\label{lem:mink1}
  Let $Y=(Y_{ij})_{ij}$ be a Minkowski-reduced $3\times 3$ positive definite
  symmetric real matrix. Then for all $n\in\R^3$ we have ${}^t n Y n \geq
  \frac{1}{100} Y_{11}\, {}^tn\, n$.
\end{lemma}
Note that by~\eqref{eq:85}, we have $\frac{1}{100}Y_{11} \geq 0.0085$. For the
theoretical complexity bound $O(\sqrt{P})$, it will suffice to use this
lemma as it is. However, for a practical algorithm, the $\frac{1}{100}Y_{11}$
is far from optimal, and we use the following better constant. Let
\begin{align*}
  c_1 &=
  \min(Y_{11}-Y_{12}-|Y_{13}|,\ Y_{22}-Y_{21}-Y_{23},\ Y_{33}-Y_{32}-|Y_{31}|),
  \quad\mbox{and} \\
  c &= \max\left(c_1, \frac{1}{100} Y_{11}\right)
  \geq \frac{1}{100} Y_{11}\geq 0.0085,
\end{align*}
which in practice tends to be much larger than $\frac{1}{100} Y_{11}$.

\begin{lemma}\label{lem:mink2}
  Let $Y=(Y_{ij})_{ij}$ be a Minkowski-reduced $3\times 3$ positive definite
  symmetric real matrix. Then for all $n\in\R^3$ we have ${}^t n Y n \geq c\,
  {}^tn\, n$.
\end{lemma}

\begin{proof}
  In case $c=\frac{1}{100} Y_{11}$, use Lemma~\ref{lem:mink1}. Otherwise, we
  have $c=c_1$. Now, for $(m,n,p) \in \mathbb{R}^3$, using the inequalities
  $2|mn| \leq (m^2+n^2)$ and $Y_{12}$, $Y_{23}\geq 0$ we have
  \begin{multline*}
    \Im \left({}^t (m,n,p) \tau (m,n,p)\right) \geq\\
    (Y_{11} - Y_{12} - |Y_{13}|) m^2 +
    (Y_{22} - Y_{21} - Y_{23}) n^2 + (Y_{33} - Y_{32} - |Y_{31}|) p^2
    \geq\\ c_1(m^2+n^2+p^2) . \qedhere
  \end{multline*}
\end{proof}

Now we prove the complexity result \eqref{eq:precisionnaive}. By
Lemma~\ref{lem:mink2}, we have
\begin{equation}
  \begin{split}
    |\theta_0(0,\tau) - S_B| & \leq 8 \sum_{\substack{m \text{ or }
      n \text{ or } p \geq B\\ \text{and } m, n, p\geq 0}} e^{-\pi c (m^2+n^2+p^2)} \\
    & \leq 24 \sum_{m \geq B, n \geq 0, p \geq 0}
      e^{-\pi c (m^2+n^2+p^2)} \\
    & \leq 24 \frac{e^{-\pi c B^2}}{(1-e^{- \pi c})^3}
    \leq \exp({14.09 - \pi c B^2}),
  \end{split}
\end{equation}
since we have an absolute lower bound $c\geq 0.0085$. Therefore taking
\begin{equation*}
  B = \sqrt{(P\log(2) + 14.09)/(\pi c)} = O(\sqrt{P})
\end{equation*}
is enough to ensure that $S_B$ is within $2^{-P}$ of $\theta_0$. This proves
our complexity estimates for Algorithm~\ref{alg:thetanullwerteNaive}, when
combined with the following deferred proof.

\begin{proof}[Proof of Lemma~\ref{lem:mink1}]
  Suppose there is an $n\in \R^3$ with ${}^tn Y n < \frac{1}{100}Y_{11}{}^t n
  n$. Let $I\in \{1,2,3\}$ be  such that $n_I^2 = \max_i n_i^2$. Let $\{I,J,K\}
  = \{1,2,3\}$. Without loss of generality, we have $Y_{II}=1$ (scale $Y$) and
  $n_I = 1$ (scale $n$). Then $|n_i|\leq n_I=1$ and $Y_{11}\leq Y_{II}=1$.

  Let $s_{ij} =s_{ji} = 1$ if $Y_{ij} \geq 0$ and $s_{ij} =s_{ji}= -1$ if
  $Y_{ij} < 0$. We get
  \begin{equation} \label{eq:mink}
    \frac{3}{100}  \geq \frac{1}{100} Y_{11} {}^t n n > {}^t n Y n
    = \sum_{i} n_i^2 (Y_{ii} - \sum_{j\not=i} |Y_{ij}|) +
      \sum_{\substack{\{i,j\}\\
      \mathrm{s.t.}\ i\not=j}} (n_i+s_{ij} n_j)^2 |Y_{ij}|.
  \end{equation}
  By \eqref{eq:minkowskiexample}, we have $|Y_{ij}|\leq \frac{1}{2} Y_{ii}$, so
  all terms on the right hand side are non-negative.

  We distinguish between three cases: \textit{I}, \textit{II$+$}
  and \textit{II$-$}.\\
  \textit{Case I:} There exists a $j\not=I$ with $s_{Ij} n_j > -\frac{3}{4}$.\\
  \textit{Case II$\pm$:} For all $j\not=I$, we have $s_{Ij} n_j \leq -\frac{3}{4}$
  and $s_{13} = \pm 1$.

  \noindent \textit{Proof in case I.}
  Without loss of generality $j=J$. We take two terms from \eqref{eq:mink}:
  \begin{align}
    0.03 &\geq (Y_{II} - |Y_{IJ}| - |Y_{IK}|) + (1+s_{IJ} n_J)^2 |Y_{IJ}|
    \nonumber \\
    &\geq Y_{II} +(-1 + (1/4)^2)|Y_{IJ}| - |Y_{IK}| \geq
    (1 -15/32-1/2)Y_{II} \geq 0.031,
  \end{align}
  which is a contradiction.

  \noindent \textit{Proof in case II$+$.}
  In this case, we have $s_{ij} = 1$ for all $i$ and $j$. In particular, we
  have $n_J, n_K \leq -\frac{3}{4}$ both negative.
  We again take two terms from \eqref{eq:mink}:
  \begin{align}
    3/100 &\geq n_J^2(Y_{JJ} - |Y_{IJ}| - |Y_{JK}|) + (n_J+n_K)^2 |Y_{JK}|
    \nonumber \\
    &\geq n_J^2(Y_{JJ} - |Y_{IJ}|)
    \geq (3/4)^2Y_{JJ}(1-\frac{1}{2}) \geq (9/32) Y_{JJ},
  \end{align}
  so $Y_{JJ} \leq 8/75$. By symmetry, we also have $Y_{KK}\leq 8/75$. Using
  \eqref{eq:mink} again, we get
  \begin{equation}
    0.03 \geq (Y_{II} - |Y_{IJ}| - |Y_{JK}|)
    \geq Y_{II} - \frac{1}{2} Y_{JJ} - \frac{1}{2} Y_{KK}
    \geq 1-8/75 > 0.89,
  \end{equation}
  which is another contradiction.

  \noindent \textit{Proof in case II$-$.}
  The proof in this case is different from the other two cases: we will show
  that $Y$ is close to
  \begin{equation}
    X = \frac{1}{2}
    \left(\begin{array}{ccc}
    \phantom{-}2 & \phantom{-}1 & -1 \\
    \phantom{-}1 & \phantom{-}2 & \phantom{-}1 \\
    -1 & \phantom{-}1 & \phantom{-}2
    \end{array}\right).
  \end{equation}
  Let $(\epsilon_{ij})_{ij} = Y - X$. We have $|n_i| = -s_{Ii} n_i \geq 3/4$
  for all $i\in\{1,2,3\}$, hence \eqref{eq:mink} gives for all
  $\{i,j,k\}=\{1,2,3\}$:
  \begin{align}
    0.06 > 0.03/n_i^2\geq (Y_{ii} - |Y_{ij}| - |Y_{ik}|)\geq \frac{1}{2} Y_{ii} - s_{ik}Y_{ik}\geq 0.
  \end{align}
  With $X_{ii} = 1$ and $X_{ij} = \frac{1}{2}s_{ij}$, this becomes
  \begin{align}\label{eq:epsilon}
    0.06 >  \frac{1}{2} \epsilon_{ii} - s_{ik}\epsilon_{ik}\geq 0.
  \end{align}
  As $\epsilon_{II} = 0$, we get $0\leq -s_{Ik}\epsilon_{Ik} < 0.06$ for all
  $k\not=I$. Applying \eqref{eq:epsilon} again, but now with $k=I$, we get
  $\frac{1}{2}|\epsilon_{ii}| < 0.06$ for all $i\not=I$. Applying
  \eqref{eq:epsilon} with $i=J$, $k=K$, we finally get $|\epsilon_{JK}|< 0.12$.
  As $Y$ is Minkowski-reduced, we have
  \begin{equation}
    \begin{split}
      1 = Y_{II} &\leq (1,-1,1) Y {}^t(1,-1,1) \\
      &=(1,-1,1) X {}^t(1,-1,1) + (1,-1,1) (\epsilon_{ij})_{ij} {}^t(1,-1,1)\\
      &
      \leq 0 + \sum_{i=1}^3\sum_{j=1}^3 |\epsilon_{ij}| < 0.72,
    \end{split}
  \end{equation}
  contradiction.
\end{proof}

\subsubsection{Fast algorithm for the Thetanullwerte}

In this section, we generalize the strategy described in genus 1 and 2
in~\cite{Dupont} and ideas taken from~\cite[Chapter~7]{Labrande}. This leads to
an evaluation algorithm with running time $O(\mathcal{M}(P)$ $\log P)$.

We start, as in~\cite{Dupont}, by writing the $\tau$-duplication formulas in
terms of $\theta_i^2$. For example, we can write,\footnotesize
\begin{equation}
  \theta_1(0,2\tau)^2
  =
  \frac{\sqrt{\theta_0^2} \sqrt{\theta_1^2} +
  \sqrt{\theta_2^2}\sqrt{\theta_3^2} +
  \sqrt{\theta_4^2} \sqrt{\theta_5^2} +
  \sqrt{\theta_6^2} \sqrt{\theta_7^2}}{4}(0,\tau).
\end{equation}
\normalsize
These formulas match the iteration used in the definition of the genus 3
\defi{Borchardt mean} $\mathcal{B}_3$~\cite{Dupont}. They can be seen as a
generalization of the arithmetic-geometric mean to higher genus, since both
involve Thetanullwerte and converge quadratically~\cite{Dupont}.

Applying the $\tau$-duplication formula to the fundamental Thetanullwerte
repeatedly gives (recall that we write $\theta_{n_1, \ldots, n_k}$ for the
$k$-tuple $\theta_{n_1}, \ldots, \theta_{n_k}$)
\begin{equation}
  \mathcal{B}_3\left(\theta_{0,1,\ldots,7}(0,\tau)^2\right) = 1
\end{equation}
assuming one picks correct square roots $\theta_i(0,2^k\tau)$ of
$\theta_i(0,2^k\tau)^2$. By the homogeneity of the Borchardt mean, we can write
\begin{equation} \label{eq:borchardthomog}
  \mathcal{B}_3\left(1,
  \frac{\theta_{1, \cdots, 7}(0,\tau)^2}{\theta_0(0,\tau)^2}\right)
  = \frac{1}{\theta_0(0,\tau)^2}.
\end{equation}
We wish to use this equality to compute the right-hand side from the quotients
of Thetanullwerte; this is a key ingredient to the quasi-linear running time of
our algorithm. The difficulty here stems from the fact that the Borchardt mean
requires a technical condition on the square roots picked at each step (``good
choice'') in order to get a quasi-linear running time, and sometimes these
choices of square roots do not correspond to the values of $\theta_i$ we are
interested in (\ie{}, would not give ${1}/{\theta_0(0,\tau)^2}$ at the end of
the procedure). We sidestep this difficulty using the same strategy
as~\cite{Labrande}: we design our algorithm so that the square roots we pick
always correspond to the values of $\theta_i$ we are interested in, even when
they do not correspond to ``good choices'' of the Borchardt mean. This slows
down the convergence somewhat; however, one can prove (using the same method as
in~\cite[Lemma~7.2.2]{Labrande}) that after a number of steps that only depends
on $\tau$ (and not on $P$), our choice of square roots always coincides with
``good choices''. After this point, only $\log P$ steps are needed to compute
the value with absolute precision $P$, since the Borchardt mean converges
quadratically; this means that the right-hand side of
Equation~\eqref{eq:borchardthomog} can be evaluated with absolute precision $P$
in $O(\mathcal{M}(P) \log P)$.

The next goal is to find a function $\mathfrak{F}$ to which we can apply
Newton's method to compute these quotients of Thetanullwerte (and, ultimately,
the Thetanullwerte). For this, we use the action of the symplectic group on
Thetanullwerte to transform \eqref{eq:borchardthomog} and get relationships
involving the coefficients of $\tau$. Using the action of the matrices
described in~\cite[Chapitre~9]{Dupont}, along with the Borchardt mean, we can
build a function $f$ with the property that
\begin{multline}
  f\left(\frac{\theta_{1, \cdots, 7}(0,\tau)^2}{\theta_0(0,\tau)^2}\right)
  =\\
  (-i\tau_{11}^{\phantom{1}}, -i\tau_{22}^{\phantom{1}},
  -i\tau_{33}^{\phantom{1}}, \tau_{12}^2 - \tau_{11}^{\phantom{1}}\tau_{22}^{\phantom{1}},
  \tau_{13}^2 - \tau_{11}^{\phantom{1}}\tau_{33}^{\phantom{1}},
  \tau_{23}^2 - \tau_{22}^{\phantom{1}}\tau_{33}^{\phantom{1}})
\end{multline}
However, the above function is a function from $\mathbb{C}^7$ to
$\mathbb{C}^6$; this is a problem, as it prevents us from applying Newton's
method directly. As discussed in~\cite[Chapter~7]{Labrande}, there are two ways
to fix this: either work on the variety of dimension 6 defined by the
fundamental Thetanullwerte, or add another quantity to the output and hope that
the Jacobian of the system is then invertible. We choose the latter solution,
and build a function $\mathfrak{F} : \mathbb{C}^7\rightarrow\mathbb{C}^7$ by
adding to the function $f$ above an extra output, equal to $-i \det(\tau)$,
which is motivated by the symplectic action of the matrix $\mathfrak{J} =
\begin{pmatrix} 0 & -I_g \\ I_g & 0
\end{pmatrix}$ on the Thetanullwerte:
\begin{equation}
  \theta_{0,1,2,3,4,5,6,7}^2(0, \mathfrak{J} \cdot \tau)
  =
  -i \det(\tau) \theta_{0,8,16,24,32,40,48,56}^2(0, \tau).
\end{equation}
The following Algorithm~\ref{alg:thetanullwerteInvertibleFunction} explicitly
defines the function $\mathfrak{F}$ that we will use.

\begin{thealgo}[Given a $7$-tuple $a_1$, $a_2, \ldots, a_7\in \C$, computes a
    number $\mathfrak{F}(a_1, \ldots, a_7)$, defined by the steps in this
    algorithm. Here we are specifically interested in the value
    $\mathfrak{F}(\,{\theta_{1, \ldots, 7}(0,\tau)^2}/{\theta_0(0,\tau)^2}\,)$,
    so for clarity we abuse notation and denote $a_i$ by
  $\theta_i(0,\tau)^2/\theta_0(0,\tau)^2$.
]\label{alg:thetanullwerteInvertibleFunction}{}\
\begin{enumerate}
  \item Compute $t_0 = \mathcal{B}_3(\,1,{\theta_{1, \ldots,
    7}(0,\tau)^2}/{\theta_0(0,\tau)^2}\,)$.
  \item Compute $t_i = ({1}/{t_0}) \times {\theta_i(0,\tau)^2} /
    {\theta_0(0,\tau)^2}$.
  \item $t_i \gets \sqrt{t_i}$, choosing the square root that coincides with
    the value of $\theta_i(0,\tau)$ (computed with low precision just to inform
    the choice of signs).
  \item Apply the $\tau$-duplication formulas to the $t_i$ to compute complex
    numbers that by abuse of notation we write as $\theta_i(0,2\tau)^2$. (Here
    if $t_i = \theta_i(0,\tau)$, then ``$\theta_i(0,2\tau)^2$'' is really equal
    to $\theta_i(0,2\tau)^2$.)
  \item $r_1 \gets \theta_{32}^2(0,2\tau) \times \mathcal{B}_3(\,1,
    {\theta_{32,33,34,35,0,1,2,3}^2(0,2\tau)}/{\theta_0^2(0,2\tau)}\,)$.
  \item $r_2 \gets \theta_{16}^2(0,2\tau) \times \mathcal{B}_3(\,1,
    {\theta_{16,17,0,1,20,21,4,5}^2(0,2\tau)}/{\theta_0^2(0,2\tau)}\,)$.
  \item $r_3 \gets \theta_{8}^2(0,2\tau) \times \mathcal{B}_3(\,1,
    {\theta_{8,0,10,2,12,4,14,6}^2(0,2\tau)}/{\theta_0^2(0,2\tau)}\,)$.
  \item $r_4 \gets \theta_{0}^2(0,2\tau) \times \mathcal{B}_3(\,1,
    {\theta_{0,1,32,33,16,17,48,49}^2(0,2\tau)}/{\theta_0^2(0,2\tau)}\,)$.
  \item $r_5 \gets \theta_{0}^2(0,2\tau) \times \mathcal{B}_3(\,1,
    {\theta_{0,32,2,34,8,40,10,42}^2(0,2\tau)}/{\theta_0^2(0,2\tau)}\,)$.
  \item $r_6 \gets \theta_{0}^2(0,2\tau) \times \mathcal{B}_3(\,1,
    {\theta_{0,16,8,24,4,20,12,28}^2(0,2\tau)}/{\theta_0^2(0,2\tau)}\,)$.
  \item $r_7 \gets \theta_0^2(0,2\tau) \times \mathcal{B}_3(\,1,
    {\theta_{0,8,16,24,32,40,48,56}^2(0,2\tau)}/{\theta_0^2(0,2\tau)}\,)$.
  \item Return $(\, {r_1}/{2},\, {r_2}/{2},\, {r_3}/{2},\,
    {r_4}/{4},\, {r_5}/{4},\, {r_6}/{4},\, {r_7}/{8} \,)$.
\end{enumerate}
\end{thealgo}

The final part of our algorithm applies Newton's method to $\mathfrak{F}$, by
starting with an approximation of the quotients of Thetanullwerte with large
enough precision $P_0$ to ensure that the method converges. In practice, we
found that a starting precision $P_0 = 450$ was on the one hand large enough to
make Newton's method converge quickly and on the other hand small enough so
that the fast algorithm does not get slowed down too much by first doing the
naive algorithm to precision~$P_0$. Since computing $\mathfrak{F}$ is
asymptotically as costly as computing the Borchardt mean, and since there is no
extra asymptotic cost when applying Newton's method if one doubles the working
precision at each step, we get an algorithm which computes the genus 3
Thetanullwerte with $P$ digits of precision with time $O(\mathcal{M}(P) \log
P)$. This algorithm was implemented in \Magma{}, along with the aforementioned
naive algorithm. For our examples, the fast algorithm always gives a result with
more than $2000$ digits of precision in less than $10$ seconds.

\subsection{Computation of the Dixmier--Ohno invariants}
\label{sec:do}

Consider Thetanullwerte $(\theta_0 (\tau), \ldots, \theta_{63} (\tau)) \in
\C^{64}$ as computed in the previous section. Then by Riemann's vanishing
theorem~\cite[V.th.5]{rauch} and Clifford's theorem~\cite[Chap.3,\S1]{harris}
the values correspond to a smooth plane quartic curve if and only if 36 of them
are non-zero. If this condition is satisfied, the following procedure
determines the equation of a plane quartic $X_{\C}$ for which there is a
Riemann matrix $\tau$ that gives these Thetanullwerte.

Using~\cite[p.108]{weber} (see also~\cite{fiorentino}), we compute the
\defi{Weber moduli}
\begin{equation} \label{eq:webermoduli}
  \begin{aligned}
    a_{11} & :=i \frac{\theta_{33} \theta_{5}}{\theta_{40} \theta_{12}}, &
      a_{12} & :=i \frac{\theta_{21} \theta_{49}}{\theta_{28} \theta_{56}}, &
      a_{13} & :=i \frac{\theta_{7} \theta_{35}}{\theta_{14} \theta_{42}}, \\
    a_{21} & :=i \frac{\theta_{5} \theta_{54}}{\theta_{27} \theta_{40}}, &
      a_{22} & :=i \frac{\theta_{49} \theta_{2}}{\theta_{47} \theta_{28}}, &
      a_{23} & :=i \frac{\theta_{35} \theta_{16}}{\theta_{61} \theta_{14}}, \\
    a_{31} & :=-\frac{\theta_{54} \theta_{33}}{\theta_{12} \theta_{27}}, &
      a_{32} & :=\frac{\theta_{2} \theta_{21}}{\theta_{56} \theta_{47}}, &
      a_{33} &:=\frac{\theta_{16} \theta_{7}}{\theta_{42} \theta_{61}}.
  \end{aligned}
\end{equation}
Note that these numbers depend on only $18$ of the Thetanullwerte. The three
projective lines $\ell_i : a_{i1} x_1+a_{i2} x_2+a_{i3} x_3=0$ in $\P^2_{\C}$,
together with the four lines
\begin{equation}
  x_1=0,\ x_2=0,\ x_3=0,\ x_1+x_2+x_3=0
\end{equation}
will form a so-called \defi{Aronhold system} of bitangents to the eventual
quartic $X_{\C}$. Considering the first three lines as a triple of points
$((a_{i1} : a_{i2} :a_{i3}))_{i=1\ldots3}$ in $(\P^2)^3$, one obtains a point
on a $6$-dimensional quasiprojective variety. Its points parametrize the moduli
space of smooth plane quartics with full level two
structure~\cite{grossharris}.

From an Aronhold system of bitangents, one can reconstruct a plane quartic
following  Weber's work~\cite[p.93]{weber} (see also~\cite{agmri, fiorentino}).
We take advantage here of the particular representative $(a_{i1}, a_{i2},
a_{i3})$ of the projective points $(a_{i1}:a_{i2}:a_{i3})$ to simplify the
algorithm presented in \loccit{} Indeed, normally that algorithm involves
certain normalization constants $k_i$. However, in the current
situation~\cite[Cor.2]{fiorentino} shows that these constants are automatically
equal to $1$ for our choices of $a_{ij}$ in \eqref{eq:webermoduli}, which leads
to a computational speedup. Let $u_1$, $u_2$, $u_3\in \C[x_1,x_2,x_3]$ be given
by
\begin{equation}
  \begin{bmatrix}
    u_1 \\ u_2 \\ u_3
  \end{bmatrix} =
  \begin{bmatrix}
    1 & 1 & 1 \\
    \frac{1}{a_{11}} & \frac{1}{a_{12}} &\frac{1}{ a_{13}} \\
    \frac{1}{a_{21}} & \frac{1}{a_{22}} & \frac{1}{a_{23}}
  \end{bmatrix}^{-1}
  \cdot
  \begin{bmatrix}
    1 & 1 & 1 \\
    a_{11} & a_{12} & a_{13} \\
    a_{21} & a_{22} & a_{23}
  \end{bmatrix}
  \cdot
  \begin{bmatrix}
    x_1 \\ x_2 \\ x_3
  \end{bmatrix}.
\end{equation}
Then $X_{\C}$ is the curve defined by the equation $(x_1 u_1+x_2 u_2-x_3
u_3)^2-4 x_1 u_1 x_2 u_2=0$.\\

We now have a complex model $X_{\C}$ of the quartic curve that we are looking
for. Note that there is no reason to expect $X_{\C}$ to be defined over $\Q$;
its coefficients will in general be complicated algebraic numbers that are
difficult to recognize algebraically. To get around this problem, we first
approximate its $13$ \defi{Dixmier--Ohno invariants}, which were defined
in~\cite{dixmier, elsenhans, giko} (see~\cite[Sec.1.2]{LRS16} for a short
description). These invariants
\begin{equation}
  \DO = (I_3 : I_6 : I_9 : J_9 : I_{12} : J_{12} : I_{15} : J_{15} : I_{18} :
         J_{18} : I_{21} : J_{21} : I_{27})
\end{equation}
are homogeneous expressions in the coefficients of a ternary quartic form.
Their degrees in the coefficients of such a form are
\begin{equation}
  \underline{d} = (3,6,9,9,12,12,15,15,18,18,21,21,27).
\end{equation}
Therefore the evaluation of these invariants at $X_{\C}$ (which we still denote
by $\DO$) gives rise to a point in the weighted projective space
$\P^{\underline{d}}$. Note that $I_{27}$ is the discriminant of $X_{\C}$, which
is non-zero.

For a ternary quartic form over $\Qbar$ that is equivalent to a ternary quartic
form over $\Q$, the tuple $\DO$ defines a $\Q$-rational point in
$\P^{\underline{d}}$. This is not to say that the entries of $\DO$ itself are
in $\Q$. However, we can achieve this by suitably normalizing this tuple. When
$I_3 \ne 0$ (as will always be the case for us), we can for instance use the
normalization
\begin{equation}
  \DO^{\textrm{norm}} = \left(1, \frac{I_6}{I_3^2}, \frac{I_9}{I_3^3},
  \frac{J_6}{I_3^3}, \frac{I_{12}}{I_3^4}, \frac{J_{12}}{I_3^4},
  \frac{I_{15}}{I_3^5}, \frac{J_{15}}{I_3^5}, \frac{I_{18}}{I_3^6},
  \frac{J_{18}}{I_3^6}, \frac{I_{21}}{I_3^7}, \frac{J_{21}}{I_3^7},
  \frac{I_{27}}{I_3^9}\right).
\end{equation}

Our program concludes by computing the best rational approximation of (the real
part) of the Dixmier--Ohno invariants $\DO^{\textrm{norm}}$ by using the
corresponding (\textsc{Pari}~\cite{pari-gp}) function
\texttt{BestApproximation} in \Magma{} at increasing precision until the
sequence stabilizes. In practice, this does not take an overly long time: we
worked with less than $1000$ decimal digits and the denominators involved never
exceeded $100$ decimal digits.

For some of the CM fields, we in fact obtain $4$ isomorphism classes of
principally polarized abelian varieties. But by
Theorem~\ref{thm:list_of_fields}, we know that exactly one of them has field of
moduli~$\Q$. Of course we do not know in advance which of the four complex tori
under consideration has this property. In such a case, we use
\texttt{BestApproximation} for each of the four cases and we observe that this
succeeds (at less than $1000$ decimal digits) for exactly one of them. We then
only set aside the Dixmier-Ohno invariants of that case for later
consideration.

Some manipulations, illustrated below with the case 15, then give us an
integral representative $\DO^{\min}$ of the Dixmier--Ohno invariants for which
the gcd of the entries is minimal. We denote this $13$-tuple by
\begin{displaymath}
  \DO^{\min} = (I_3^{\min},\, I_6^{\min},\, I_9^{\min},\, \ldots ,\, J_{21}^{\min} ,\, I_{27}^{\min})\,.
\end{displaymath}

\begin{example}
  In case 15, the approximation that we obtain is
  \begin{equation}
    \DO^{\textrm{norm}}=
    \left(1
    :
    \frac{3967}{609408}
    : \cdots  :
    \frac{346304226226660371}{1980388294678257795596288}\right).
  \end{equation}
  We first get an integral representative by taking $\lambda$ to be the least
  common multiple of the denominators of $\DO^{\textrm{norm}}$ and setting
  $\DO'$ be equal to
  {\small
    \begin{equation*}
      (\lambda , \lambda^2 I_6 , \lambda^3 I_9 , \lambda^3 J_9 , \lambda^4
      I_{12} , \lambda^4 J_{12} , \lambda^5 I_{15} , \lambda^5 J_{15} ,
      \lambda^6 I_{18} , \lambda^6 J_{18} , \lambda^7 I_{21} , \lambda^7
      J_{21} , \lambda^9 I_{27}).
    \end{equation*}
  }
  We can now find the prime factors $p$ of $I_3$ and look at the valuations at
  $p$ of each entry of $\DO'$. Since for an invariant $I$ of degree $3 n$, we
  have that
  \begin{equation}
    I\left(p F\left(\frac{x}{p},y,z\right)\right)
    =
    p^{3 n} I\left(F\left( \frac{x}{p},y,z\right)\right)
    =
    p^{3n} p^{-4 n} I(F) = \frac{I(F)}{p^n}
  \end{equation}
  by this procedure, we can reduce the valuations at $p$ of these invariants.
  Applying this as much as possible while preserving positive valuation, we
  find
  \begin{equation}
    \DO^{\min} = (2^5 \cdot 3 \cdot 23 :  2^3 \cdot 3967 : 2^3 \cdot 3 \cdot 5
      \cdot 41 \cdot 173 \cdot 19309  :\cdots : 2^{5} \cdot 3^{27} \cdot
      19^{7}).
  \end{equation}
  Note that we cannot always get a representative with coprime entries; already
  in the case under consideration the prime $2$ divides all the entries).
\end{example}

\section{Optimized reconstruction}
\label{sec:rec}

Having the Dixmier--Ohno invariants at our disposal, it remains to reconstruct
a corresponding plane quartic curve $X$ over~$\Q$. It was indicated
in~\cite{LRS16} how such a reconstruction can be obtained; however, the
corresponding algorithms, the precursors of those currently
at~\cite{LRS16-Code}, were suboptimal in several ways. To start with, they
would typically return a curve over a quadratic extension of the base field,
without performing a further Galois descent. Secondly, the coefficients of these
reconstructed models were typically of gargantuan size. In this section we
describe the improvements to the algorithms, incorporated in the present
version of~\cite{LRS16-Code}, that enabled us to obtain the simple equations in
this paper.

The basic ingredients are the following. A Galois descent to the base field can be
found by determining an isomorphism of $X$ with its conjugate and applying an
effective version of Hilbert's Theorem 90, as was also mentioned
in~\cite{LRS16}. After this, a reduction algorithm can be applied, based on
algorithms by Elsenhans~\cite{elsenhans-good} and Stoll~\cite{Stoll2011} that
have been implemented and combined in the \Magma{} function
\texttt{Minimize\-Reduce\-Plane\-Quartic}. However, applying these two steps
concurrently is an overly naive approach, since the Galois descent step blows up the
coefficients by an unacceptable factor. We therefore have to look under the
hood of our reconstruction algorithms and use some tricks to optimize them.

Recall from~\cite{LRS16} that the reconstruction algorithm finds a quartic form
$F$ by first constructing a triple $(b_8, b_4, b_0)$ of binary forms of degree
$8$, $4$ and $0$. Our first step is to reconstruct the form $b_8$ as
efficiently as possible. This form is reconstructed from its Shioda invariants
$\SH$, which are algebraically obtained from the given Dixmier--Ohno
invariants $\DO^{\min}$. Starting from the invariants $\SH$, the
methods of~\cite{LR11} are applied, which furnish a conic $C$ and a quartic $H$
in $\P^2$ that are both defined over $\Q$. This pair corresponds to $b_8$ in
the sense that over $\Qbar$ the divisor $C \cap H$ on $C$ can be transformed
into the divisor cut out by $b_8$ on $\P^1$. A priority in this reconstruction
step is to find a conic $C$ defined by a form whose discriminant is as small as
possible.

\subsection{Choosing the right conic for Mestre reconstruction}
\label{sec:choosing-right-conic}

Let $k$ be a number field whose rings of integers $\O_k$ admits an effective
extended GCD algorithm, which is for example the case when $\O_k$ is a
Euclidean ring. We indicate how over such a field we can improve the algorithms
developed to reconstruct a hyperelliptic curve from its Igusa or Shioda
invariants in genus 2 or genus 3 respectively~\cite{mestre, LR08, LR11}.

Recall that Mestre's method for hyperelliptic reconstruction is based on
Clebsch's identities~\cite[Sec.2.1]{LR11}. It uses three binary covariants $q
= (q_1,q_2,q_3)$ of order $2$.
From these forms, one can construct a plane conic $C_q : \sum_{1 \leq i,j \leq
  3} A_{i,j} x_i x_j=0$ and a degree $g+1$ plane curve $H_q$ over the ring of
invariants. Here $g$ is the genus of the curve that we wish to reconstruct.

Given a tuple of values of hyperelliptic invariants over $k$, we can substitute
to obtain a conic and a curve that we again denote by $C_q$ and $H_q$.
Generically, one then recovers a hyperelliptic curve $X$ with the given
invariants by constructing the double cover of $C_q$ ramified over $C_q \cap
H_q$. Because the coefficients of the original universal forms $C_q$ and $H_q$
are invariants of the same degree, the substituted forms will be defined over $k$.

Finding a model of $X$ of the form $y^2=f(x)$ over $k$ (also called a
\defi{hyperelliptic model}) is equivalent to finding a $k$-rational point on
the conic $C_q$ by~\cite{LR11,lrs}. Algorithms to find such a rational point
exist~\cite{simon, voight} and their complexity is dominated by the time spent
to factorize the discriminant of an integral model of $C_q$. While a
hyperelliptic model may not exist over $k$, it can always be found over some
quadratic extension of $k$. It is useful to have such an extension given by a
small discriminant, which is in particular the case when $C_q$ has small
discriminant. Accordingly, we turn to the problem of minimizing $\disc (C_q)$.

In order to do so, we use a beautiful property of Clebsch's identities.
By~\cite[Sec.2.1.(5)]{LR11}, we have that
\begin{equation}\label{eq:discvsr}
  \disc(C_q) = \det((A_{i,j})_{1 \leq i,j \leq 3}) = R_q^2/2
\end{equation}
where $R_q$ is the determinant of $q_1,q_2,q_3$ in the basis $x^2,x z,z^2$. If
$q_3'$ is now another covariant of order $2$, we can consider the \emph{family}
of covariants $q_{\lambda,\mu} = (q_1,q_2,\lambda q_3 + \mu q_3')$, $\lambda,
\mu \in k$. For this family, the multilinearity of the determinant shows
that
\begin{equation}
  R_{q_{\lambda,\mu}} = \lambda R_{(q_1,q_2,q_3)} + \mu R_{(q_1,q_2,q_3')} .
\end{equation}
The values $R_{(q_1,q_2,q_3)}$ and $R_{(q_1,q_2,q_3')}$ are invariants that can
be effectively computed and which are generically non-zero. (If either of these
invariants is zero, then, one can usually take different covariants $q_i$; if
all of these fail to give non-zero values, then typically $X$ has large reduced
automorphism group and other techniques can be used.) The key point is that we
can minimize the value of $R_{q_{\lambda,\mu}}$, and by \eqref{eq:discvsr} the
value of $\disc (C_{q_{\lambda,\mu}})$ with it, by using the extended Euclidean
algorithm to minimize the combined linear contribution of $\lambda$ and $\mu$
to the linear expression $R_{q_{\lambda,\mu}}$. This allows us to reduce the
discriminant all the way to $\gcd(R_{(q_1,q_2,q_3)}, R_{(q_1,q_2,q' _3)})$ or
beyond.

Note that we do not have $C_{q_{\lambda,\mu}} = \lambda^2 C_{(q_1,q_2,q_3)}
+\mu^2 C_{(q_1,q_2,q_3')}$. However, the coefficients of the family of conics
$C_{q_{\lambda,\mu}}$ and of $H_{q_{\lambda}}$ can be quickly found in terms of
the invariants and $\lambda,\mu$ by using the same interpolation techniques as
in~\cite[Sec.~2.3]{LR11}.

\subsection{Reconstruction of a plane quartic model from the invariants}

With these precomputations out of the way, we now search for a binary octic
form $b_8$ whose Shioda invariants come from the first step of the
reconstruction algorithm of~\cite{LRS16} applied to the Dixmier--Ohno
invariants of case 15 (\textit{cf.} Table~\ref{table: CMfields}). Except for
this case and case 6, all the other cases give conics $C$ with no rational
point and as such, a Galois descent phase is needed to find a rational quartic
(\textit{cf.} Section~\ref{sec:descent-minimization}). Case 15 is the easiest
CM plane quartic that we have to reconstruct. From the associated Shioda
invariants, we compute an invariant $R_{q_{\lambda,\mu}}$. By using the
extended GCD algorithm and substituting the result for $\lambda$ and $\mu$, we
are left with $R_q$ equal to the left hand side coefficient $2^{61}\cdot 3^{18}
\cdots  201049$\,. This factor is almost equal to the Dixmier--Ohno invariant
$I_{12}^{\min}$, the discriminant of the covariant used in our quartic reconstruction.
Indeed, the considerations in~\cite{LRS16} show that our reconstruction
algorithm fails when $I_{12}^{\min} = 0$, and more precisely that this failure
occurs when trying to reconstruct $b_8$ via Mestre's method. Hence the primes
which divide $I_{12}^{\min}$ naturally appear in the discriminants of
$C_{q_{\lambda,\mu}}$. A substantially smaller $R_q$ cannot therefore be
expected.

Now as we know the factorization of $R_q$, we can efficiently determine if
the conic $C_{q_{\lambda,\mu}}$ has a rational point. Unexpectedly, it has
one, and after a change of variable we map it to the point $(1:1:0)$. The
conic is then
\begin{footnotesize}
  \begin{math}
    C = x^2 - y^2 - z^2
  \end{math}
\end{footnotesize}
and the corresponding quartic $H$ has approximately 50-digit coefficients.

Finally, it remains to compute the geometric intersection $C \cap H$. This
yields the octic $b_8$. The forms $b_0$ and $b_4$ computed by the plane quartic
reconstruction algorithm~\cite{LRS16} are therefore defined over $\Q$ as well.
By applying the linear map $(\ell^*)^{-1}$ defined in \loccit, we get a plane
quartic defined over $\Q$ too. It remains to reduce the size of its
coefficients as explained in Section~\ref{sec:descent-minimization} to obtain
the equation given in Section~\ref{sec:results}.

\subsection{Galois descent and minimization}
\label{sec:descent-minimization}

Now suppose that we have in this way found a pair $(C, H)$ as above, for which
$C$ has minimal discriminant. We can then further optimize this pair by
applying the following two steps:
\begin{enumerate}
  \item Minimize the defining equation of $C$ by using the theory of quaternion
    algebras (implemented in the \Magma{} function \texttt{MinimalModel});
  \item Apply the reduction theory of point clusters~\cite{Stoll2011} applied
    to the intersection $C \cap H$ (implemented in the \Magma{} function
    \texttt{Reduce\-Cluster}).
\end{enumerate}

The second step above is more or less optional; typically it leads to a rather
better $H$ at the cost of a slightly worse $C$. Regardless, at the end of this
procedure, we can construct a binary form $b_8$ over a quadratic extension $K$
of $\Q$ by parametrizing the conic $C$, and we then reconstruct $b_4$ and $b_0$
as in~\cite{LRS16}. The associated ternary quartic form $F$ is usually defined
over a quadratic extension of $\Q$. Since its covariant $\rho (F)$
from~\cite{LRS16} is a multiple of $y^2 - x z$, we can immediately apply the
construction from~\cite{VanRijnswou} to obtain an element $[ M ] \in \PGL_3
(K)$ that up to a scalar $\lambda$ transforms $F$ into its conjugate $\sigma
(F)$:
\begin{equation}\label{eq:2}
  [ \sigma (F) ] = [ F . M ].
\end{equation}
In the cases under consideration we know that the curve defined by $F$ descends
because of the triviality of its automorphism group. This implies that the
cocycle defined by the class $[ M ]$ lifts to $\GL_3 (K)$. Explicitly, let $M
\in \GL_3 (K)$ be some representative of the class $[ M ]$. Then we have
\begin{equation}\label{eq:3}
  M \sigma (M) = \pi
\end{equation}
for some scalar matrix $\pi$. Conjugating this equality shows that in fact $\pi
\in \Q$, and taking determinants yields $\delta \sigma (\delta) = \pi^3$, where
$\delta$ is the determinant of~$M$. Now let $M_0 = (\pi / \delta) M$. Then we
have
\begin{equation}
  M_0 \sigma (M_0)
  = \frac{\pi}{\delta} M \frac{\sigma (\pi)}{\sigma (\delta)} \sigma (M)
  = \frac{\pi \sigma (\pi) \pi}{\delta \sigma (\delta)}
  = \frac{\pi^3}{\delta \sigma (\delta)} = 1 .
\end{equation}

We may therefore assume that $M \in GL_3 (K)$ corresponds to a lifted cocycle.
The Galois cohomology group $H^1 (\Gal (K \ext \Q), \GL_3 (K))$ is trivial;
Hilbert's Theorem 90 can be used to construct a coboundary $N$ for $M$, that
is, a matrix in $\GL_3 (K)$ for which
\begin{equation}
  M \sigma (N) = N .
\end{equation}
After choosing a random matrix $R \in \GL_3 (K)$, one can in fact take
\begin{equation}\label{eq:1}
  N = R + M \sigma (R) .
\end{equation}

We thus obtain a coboundary $N$ corresponding to the cocycle $M$. If we put
$F_0 = F . N$, then the class $[ F_0 ]$ is defined over $\Q$. The transformed
form $F_0$ itself still need not be defined over $\Q$, but this can be achieved
by dividing it by one of its coefficients.

A complication is that the determinant of a random matrix $N$ as
in~\eqref{eq:1} typically has a rather daunting factorization. These factors
can (and usually will) later show up as places of bad reduction of the
descended form $F_0$. It is therefore imperative to avoid a bad factorization
structure of the determinant of $N$. This, however, can be ensured by
performing a lazy factorization of this determinant and passing to a next
random choice if the result is not satisfactory.

After we have obtained a form $F_0$, one can apply the \Magma{} function
\texttt{Minimize\-Reduce\-Plane\-Quartic}; this function combines a
discriminant minimization step due to Elsenhans in~\cite{elsenhans-good} with
the reduction theory of Stoll in~\cite{Stoll2011}. Typically the first of these
steps leads to the most significant reduction of the coefficient size, since it
applies a suitable transformation in $\GL_3 (\Q)$ whose determinant is a large
prime, whereas the cluster reduction step is a further optimization involving
only the subgroup $\SL_3 (\Z)$. As mentioned above, we can save some time in
the minimization step by carrying over the primes in the factorization of the
determinant of the coboundary $N$, since these will recur in the set of bad
primes of $F_0$.\bigskip

All in all, we get the following randomized algorithm whose heuristic
complexity is polynomial in the size of the Dixmier--Ohno invariants, if we
assume that the factorizations of $I_{12}^{\min}$ and $I_{27}^{\min}$ are
known, and that $\det N$ behaves as a random integer.

\begin{thealgo}
  [Integral plane quartic reconstruction from its Diximier-Ohno invariants
  $\DO^{\min}$ when the factorizations of $I_{12}^{\min}$ and $I_{27}^{\min}$ are known]
  \label{alg:Qreconstruct}\ \\
  \begin{enumerate}
  \item Repeat the following steps until $N\ne0$ and the full factorization of
    $\det(N)$ is known.
    \begin{enumerate}
    \item Calculate the Shioda invariants $\SH$ of $b_8$ \emph{(as
      explained in~\cite{LR11})}.
    \item Evaluate the conic $C_{q_{\lambda,\mu}}$ at $\SH$ and
      determine $(\lambda, \mu)$ by using the extended Euclidean algorithm
      \emph{(so that $\disc C_{q_{\lambda,\mu}} \simeq I_{12}^{\min}$, see
      Section~\ref{sec:choosing-right-conic})}.
    \item Choose a point $P$ on the conic $C_{q_{\lambda,\mu}}$ and use it to
      parametrize the conic.\\
      \emph{(To achieve this, let $P$ to be any rational point of
      $C_{q_{\lambda,\mu}}$ if it is easy to find. Otherwise intersect $C_q$ with a
      random rational line with a defining equation of small height and let $P$
      be the quadratic point defined by this intersection.)}
    \item Intersect $C_{q_{\lambda,\mu}}$ and $H_{q_{\lambda,\mu}}$ to obtain
      the octic $b_8$, then calculate the forms $b_4$ and $b_0$ and
      reconstruct a quartic $F$ via the map~$\ell^*$.
    \item If $F$ is defined over $\Q$ then set $N$ to be the identity matrix of
      $\GL(3,\Q)$, else let $N$ be a random coboundary as in~\eqref{eq:1}.
    \item Try to compute a factorization of $\det N$. If this fails within the
      allocated time, then start over.
    \end{enumerate}
  \item Let $F_0 = F.N$ and divide the result by one of its coefficients, so
    that $F_0$ has coefficients in~$\Q$.
  \item Reduce the coefficient size of $F_0$ \emph{(with
      \texttt{\textup{MinimizeReducePlaneQuar\-tic}}, using the prime factors of
      $\det N$ and $I_{12}^{\min}$).}
  \end{enumerate}
\end{thealgo}

One important practical speedup for Algorithm~\ref{alg:Qreconstruct} exploits
that if we take $R$ in \eqref{eq:1} to be integral, then the determinants of
the random coboundaries that we compute in Step $(i)(e)$ share the same
denominator, namely that of $\pi/\delta$, where $\pi$ is as in~\eqref{eq:3} and
where $\delta$ is the determinant of $M$. In turn, the quantity $\pi/\delta$
only depends on the choice of the random line in Step $(i)(c)$. A
straightforward optimization is thus to loop over the Steps $(i)(a)-(d)$ until
a lazy factorization of the denominator of $\det(1 + M)$ yields its full
factorization (note that here $M$ is the cocycle defined by
equation~\eqref{eq:2} and $1 + M = R + M\sigma(R)$ for $R$ the identity
matrix). Once done, we can loop over the Steps $(i)(e)-(i)(f)$ to test as many
coboundaries $N = R + M\sigma(R)$ from random integral matrices $R$ as needed,
once more until the lazy factorization of the denominator of $\det N$ is its
full factorization.

In the most difficult case, \ie, case 16 (\textit{cf.}
Table~\ref{table: CMfields}), the candidates for $\det N$ have approximately
500-digit denominators and 700-digit numerators. If we allow less than a second
for the lazy factorization routine in \Magma{}, then the total computation in
the end takes less than 5 minutes on a laptop. In this case, the descended form
$F_0 = F.N$ has 1500-digit coefficients! Once the discriminant minimization
steps from~\cite{elsenhans-good} are done for each prime divisor of $\det N$,
we are left with a form that ``merely'' has 50-digit coefficients. Stoll's
reduction method~\cite{Stoll2011} then finally yields the 15-digit equation
given in Section~\ref{sec:results}.

\begin{remark}
  Bouyer and Streng~\cite[Algorithm~4.8]{BouyerStreng} show how one can avoid
  factoring in the discriminant minimization of binary forms. Such a trick
  enabled them to eliminate the need for a loop like that in Step (i) of
  Algorithm~\ref{alg:Qreconstruct} when considering curves of genus $2$. It
  remains to be seen whether a similar trick applies to Elsenhans's
  discriminant minimization of plane quartics~\cite{elsenhans-good}. If it
  does, then that would greatly speed up the reconstruction.
\end{remark}

\section{Remarks on the results}
\label{sec:remarks}

Our (heuristic) results can be found in the next section; here we discuss some
of their properties and perform a few sanity checks. The very particular
pattern of the factorization of the discriminants is already a good indicator
of the correctness of our computations. Note that as the Dixmier--Ohno
invariants that we use involve denominators with prime factors in $\left\{
2,3,5,7 \right\}$, we will not look at the valuation of our invariants at these
primes.

As was mentioned in the introduction, one of the motivations for computing this
list of curves was to have examples in hand to understand the possible
generalization of the results of Goren--Lauter~\cite{GorenLauter} in genus $2$
to non-hyperelliptic curves of genus~$3$. In genus~$2$, all primes dividing the
discriminant are primes of bad reduction for the curve. This bad reduction
provides information on the structure of the endomorphism ring of the reduction
of the Jacobian. This particular structure allows one, with additional work, to
bound the primes dividing the discriminant. Taking this even further allowed
Lauter--Viray~\cite{viray} to find out exactly which prime powers divide the
discriminant. Similar bounds on primes dividing invariants have been obtained
by K{\i{}}l{\i{}}{\c{c}}er--Lauter--Lorenzo--Newton--Ozman--Streng for
hyperelliptic~\cite{KLLNOS} and Picard~\cite{KLLNOS, KLS} curves.

Beyond these cases, so for ``generic'' genus $3$ CM curves $X$, the situation is
more involved. Let us fix terminology for a prime by calling it
\begin{enumerate}
  \item a \emph{potentially plane prime} if, after extending the base field,
    $X$ has good non-hyperelliptic reduction at this prime;
  \item a \emph{potentially hyperelliptic prime} if, after extending the base
    field, $X$ has good hyperelliptic reduction at this prime;
  \item a \emph{geometrically bad prime} in the remaining cases.
\end{enumerate}
The first case can be detected easily, but distinguishing the second from the
third case from the knowledge of the Dixmier-Ohno invariants is a difficult
task and will be the main result of \cite{LLR17}. Applying these forthcoming
results to the list of curves of Section 5, it can be proved for those curves that all primes $p>7$
dividing $I_{27}^{\min}$ with exponents $7$ and $14$ are potentially hyperelliptic whereas
the few primes that are not of this kind are geometrically bad. Primes $p>7$
dividing $\disc X\,/\,I_{27}^{\min}$ for the curves of Section 5 are all potentially plane primes.

This profusion of hyperelliptic primes is typical of the CM case. Since the
curves that we consider are CM curves, their Jacobian has potentially good
reduction at all primes. Therefore, a prime is bad for $X$ if and only if the
Jacobian of $X$ reduces to a product of two abelian sub-varieties with a
decomposable principal polarization. The locus of such abelian threefolds is of
codimension $2$ in the moduli space of principally polarized abelian
threefolds, whereas the locus of Jacobians of hyperelliptic curves has
codimension $1$. We therefore expect that ``most'' of the non-potentially-plane primes
dividing the discriminant of a CM plane quartic are potentially hyperelliptic primes.

It should be mentioned that the results of \cite{LLR17} do not provide a closed
formula for the potentially hyperelliptic primes simply in terms of the CM-type and
polarization. In fact we wish to conclude this section with two remarks on the
primes dividing $I_{27}^{\min}$ that suggest that new phenomena occur for
potentially hyperelliptic primes of plane quartics that do not have an exact equivalent in
lower genus and that will require new theoretical developments in order to be
fully explained. First, unlike the factorization pattern of the discriminants
in the genus-$2$ CM, hyperelliptic and Picard cases, the factorization pattern
of the product $b$ of the potentially hyperelliptic primes seems to fit with that of a
random integer of size~$b$. For example, in case 16 below we have $b = 19 \cdot
37\cdot 79\cdot 13373064392147$. Secondly, the following proposition (applied
for instance to $X_9$ at the primes $233$ and $857$ which are both totally
split) shows that the reduction of the Jacobian of $X$ at a potentially hyperelliptic prime
can still be absolutely simple.

\begin{proposition}\label{prop:reduction}
  Let $A$ be an abelian variety over a number field $k$ and suppose that $A$
  has CM by $\mathcal{O}_K$ for a sextic cyclic CM field~$K$. Let
  $\mathfrak{p}\subset\mathcal{O}_k$ be a prime lying over a rational
  prime~$p$. Let $n$ be the number of prime factors of $p\mathcal{O}_K$.

  Then possibly after extending $k$, the following holds for the reduction
  $\overline{A}$ of $A$ modulo~$\mathfrak{p}$.

  We have $\overline{A}\sim B^d$ where $B$ is absolutely simple and
  \begin{enumerate}
    \item If $n=2$, then $d=1$, $\overline{A}=B$ is absolutely simple, and
      $\mathrm{End}(\overline{A}_{\Fpbar})\otimes\Q$ is a central simple
      division algebra of reduced degree $3$ over the imaginary quadratic
      subfield $K_1$ of $K$. It is ramified exactly at the two primes over $p$
      of $K_1$.
    \item If $n=6$, then $d=1$, $\overline{A}=B$ is absolutely simple, and
      $\mathrm{End}(\overline{A}_{\Fpbar})\otimes\Q\cong K$.
    \item In all other cases, we have $d=3$, $\overline{A}$ is supersingular,
      $\mathrm{End}(B_{\Fpbar})\otimes\Q$ is the quaternion algebra
      $B_{p,\infty}$ over $\Q$ ramified only at $p$ and infinity, and
      $\mathrm{End}(\overline{A}_{\Fpbar})\otimes\Q$ is the $3\times 3$ matrix
      algebra over $B_{p,\infty}$.
  \end{enumerate}
  If $A$ is the Jacobian of a curve, then in cases (i) and (ii) the curve has
  potentially good reduction. (In case (iii) both good and bad reduction can
  occur.)
\end{proposition}

\begin{proof}
  By a theorem of Serre and Tate~\cite{SerreTate}, the abelian variety $A$ has
  potentially good reduction. Extend $k$ so that it has good reduction and so
  that $k$ contains the reflex field. The Shimura--Taniyama
  formula~\cite[Theorem~1(ii) in Section~13.1]{shimura61:_compl_abelian} then
  gives a formula for the Frobenius endomorphism of the reduction as an element
  $\pi \in \mathcal{O}_K$ up to units. A theorem in Honda--Tate
  theory~\cite[Th\'eor\`eme~1]{tate2} then gives a formula for the endomorphism
  algebra in terms of this~$\pi$. We did the computation for all possible
  splitting types of a prime in a cyclic sextic number field and found the
  above-mentioned endomorphism algebras over some finite extension of~$\F_p$.
  Moreover, we found that the endomorphism algebra from \loccit{}~in our cases
  does not change when taking powers of $\pi$ (\ie{}, extending $k$ and the
  extension of $\F_p$ further), so that these are indeed the endomorphism
  algebras over~$\Fpbar$.

  Finally, suppose further that $A=J(X)$ and $X$ does not have potentially good
  reduction. Then by~\cite[Corollary~4.3]{BCLLMNO}, we get that the reduction
  of $A$ is not absolutely simple, which gives a contradiction in cases (i)
  and~(ii).
\end{proof}

\section{Defining equations}
\label{sec:results}

We now give the equations of the plane quartics that we obtained for the CM
field listed in Table~\ref{table: CMfields}. The expressions of the invariants
of the curves obtained are too unwieldy to be written down completely; in fact
in some cases it is even difficult to factor all of them. Here we only show
the factorizations of $I_{27}^{\min}$ (the full list is available
at~\cite{fullistofcurves}).\medskip

\begin{dgroup*}[style={\footnotesize},spread={-2pt}]
  \begin{dmath*}
    {\Xb_{1}}:    -4169\,{{x}}^{4}-956\,{{x}}^{3}{y}+7440\,{{x}}^{3}{z}+55770\,{{x}}^{2}{{y}}^{2}+43486\,{{x}}^{2}{y}\,{z}+42796\,{{x}}^{2}{{z}}^{2}-38748\,{x}\,{{y}}^{3}-30668\,{x}\,{{y}}^{2}{z}+79352\,{x}\,{y}\,{{z}}^{2}-162240\,{x}\,{{z}}^{3}+6095\,{{y}}^{4}+19886\,{{y}}^{3}{z}-89869\,{{y}}^{2}{{z}}^{2}-1079572\,{y}\,{{z}}^{3}-6084\,{{z}}^{4}{=0}
  \end{dmath*}
  \begin{dsuspend}
    with {\footnotesize $\disc X_{1} = %
      2^{-27} \cdot 3^{-27} \cdot 13^{18} \cdot I_{27}^{\min}$}
    where
    {\footnotesize
      \begin{math}
        I_{27}^{\min} = %
        -2^{57}\cdot3^{27}\cdot5^{12}\cdot7^{9}\cdot37^{14}\cdot15187^{14}\,.
      \end{math}
    }
  \end{dsuspend}
  \begin{dsuspend}
  \end{dsuspend}
  \begin{dmath*}
    {\Xb_{2}}:    19\,{{x}}^{4}+80\,{{x}}^{3}{y}-54\,{{x}}^{3}{z}-24\,{{x}}^{2}{{y}}^{2}-34\,{{x}}^{2}{y}\,{z}+77\,{{x}}^{2}{{z}}^{2}-88\,{x}\,{{y}}^{3}-28\,{x}\,{{y}}^{2}{z}+38\,{x}\,{y}\,{{z}}^{2}+516\,{x}\,{{z}}^{3}+30\,{{y}}^{4}-36\,{{y}}^{3}{z}-135\,{{y}}^{2}{{z}}^{2}+452\,{y}\,{{z}}^{3}+4\,{{z}}^{4}{=0}
  \end{dmath*}
  \begin{dsuspend}
    with {\footnotesize $\disc X_{2} = %
      2^{-27} \cdot 3^{-27} \cdot I_{27}^{\min}$}
    where
    {\footnotesize
      \begin{math}
        I_{27}^{\min} = %
        2^{29}\cdot3^{35}\cdot7^{9}\cdot701^{14}\,.
      \end{math}
    }
  \end{dsuspend}
  \begin{dsuspend}
  \end{dsuspend}
  \begin{dmath*}
    {\Xb_{3}}:    -1210961\,{{x}}^{4}+5202144\,{{x}}^{3}{y}+408700\,{{x}}^{3}{z}-2479108\,{{x}}^{2}{{y}}^{2}+1908050\,{{x}}^{2}{y}\,{z}+8367272\,{{x}}^{2}{{z}}^{2}-4393072\,{x}\,{{y}}^{3}-6944000\,{x}\,{{y}}^{2}{z}+6772756\,{x}\,{y}\,{{z}}^{2}+10594064\,{x}\,{{z}}^{3}+4978166\,{{y}}^{4}-8342100\,{{y}}^{3}{z}+4611839\,{{y}}^{2}{{z}}^{2}+14080572\,{y}\,{{z}}^{3}-1387684\,{{z}}^{4}{=0}
  \end{dmath*}
  \begin{dsuspend}
    with {\footnotesize $\disc X_{3} = %
      -2^{-27} \cdot 3^{-18} \cdot 31^{18} \cdot I_{27}^{\min}$}
    where
    {\footnotesize
      \begin{math}
        I_{27}^{\min} = %
        2^{29}\cdot3^{36}\cdot5^{36}\cdot7^{7}\cdot233^{14}\cdot356399^{14}\,.
      \end{math}
    }
  \end{dsuspend}
  \begin{dsuspend}
  \end{dsuspend}
  \begin{dmath*}
    {\Xb_{5}}:    115\,{{x}}^{4}-766\,{{x}}^{3}{y}-1336\,{{x}}^{3}{z}+1205\,{{x}}^{2}{{y}}^{2}+5178\,{{x}}^{2}{y}\,{z}+4040\,{{x}}^{2}{{z}}^{2}+8216\,{x}\,{{y}}^{3}+1322\,{x}\,{{y}}^{2}{z}-9484\,{x}\,{y}\,{{z}}^{2}+1144\,{x}\,{{z}}^{3}-8094\,{{y}}^{4}+9032\,{{y}}^{3}{z}+9669\,{{y}}^{2}{{z}}^{2}-6292\,{y}\,{{z}}^{3}-4706\,{{z}}^{4}{=0}
  \end{dmath*}
  \begin{dsuspend}
    with {\footnotesize $\disc X_{5} = %
      2^{-27} \cdot 3^{-27} \cdot 13^{18} \cdot I_{27}^{\min}$}
    where
    {\footnotesize
      \begin{math}
        I_{27}^{\min} = %
        2^{29}\cdot3^{51}\cdot7^{7}\cdot37^{14}\cdot127^{14}\,.
      \end{math}
    }
  \end{dsuspend}
  \begin{dsuspend}
  \end{dsuspend}
  \begin{dmath*}
    {\Xb_{6}}:    1444\,{{x}}^{4}-3134924\,{{x}}^{3}{y}+5002016\,{{x}}^{3}{z}+2321857\,{{x}}^{2}{{y}}^{2}+2257732\,{{x}}^{2}{y}\,{z}+1585968\,{{x}}^{2}{{z}}^{2}-3166884\,{x}\,{{y}}^{3}+6283512\,{x}\,{{y}}^{2}{z}+1014570\,{x}\,{y}\,{{z}}^{2}-4791852\,{x}\,{{z}}^{3}+3312514\,{{y}}^{4}-7211392\,{{y}}^{3}{z}+19540084\,{{y}}^{2}{{z}}^{2}-10746888\,{y}\,{{z}}^{3}+4167513\,{{z}}^{4}{=0}
  \end{dmath*}
  \begin{dsuspend}
    with {\footnotesize $\disc X_{6} = %
      2^{-27} \cdot 3^{-27} \cdot 19^{18} \cdot I_{27}^{\min}$}
    where
    {\footnotesize
      \begin{math}
        I_{27}^{\min} = %
        2^{29}\cdot3^{51}\cdot7^{7}\cdot17^{12}\cdot127^{14}\cdot211^{14}\cdot20707^{14}\,.
      \end{math}
    }
  \end{dsuspend}
  \begin{dsuspend}
  \end{dsuspend}
  \begin{dmath*}
    {\Xb_{7}}:    -133225\,{{x}}^{4}-68935944\,{{x}}^{3}{y}+92175713\,{{x}}^{3}{z}-21721369\,{{x}}^{2}{{y}}^{2}+2990226\,{{x}}^{2}{y}\,{z}+86699691\,{{x}}^{2}{{z}}^{2}+18547032\,{x}\,{{y}}^{3}+37568944\,{x}\,{{y}}^{2}{z}+108649086\,{x}\,{y}\,{{z}}^{2}-259362054\,{x}\,{{z}}^{3}+35272208\,{{y}}^{4}+266781024\,{{y}}^{3}{z}+140110856\,{{y}}^{2}{{z}}^{2}-1192622568\,{y}\,{{z}}^{3}+173418831\,{{z}}^{4}{=0}
  \end{dmath*}
  \begin{dsuspend}
    with {\footnotesize $\disc X_{7} = %
      2^{-27} \cdot 3^{-18} \cdot 7^{18} \cdot 73^{18} \cdot I_{27}^{\min}$}
    where
    {\footnotesize
      \begin{math}
        I_{27}^{\min} = %
        -2^{29}\cdot3^{36}\cdot5^{9}\cdot7^{7}\cdot71^{14}\cdot83^{12}\cdot17665559^{14}\,.
      \end{math}
    }
  \end{dsuspend}
  \begin{dsuspend}
  \end{dsuspend}
  \begin{dmath*}
    {\Xb_{8}}:    11\,{{x}}^{4}-8\,{{x}}^{3}{y}-46\,{{x}}^{3}{z}+216\,{{x}}^{2}{{y}}^{2}+306\,{{x}}^{2}{y}\,{z}+1636\,{{x}}^{2}{{z}}^{2}-144\,{x}\,{{y}}^{3}+304\,{x}\,{{y}}^{2}{z}+15726\,{x}\,{y}\,{{z}}^{2}+7963\,{x}\,{{z}}^{3}-428\,{{y}}^{4}+6840\,{{y}}^{3}{z}-32779\,{{y}}^{2}{{z}}^{2}-16901\,{y}\,{{z}}^{3}+106789\,{{z}}^{4}{=0}
  \end{dmath*}
  \begin{dsuspend}
    with {\footnotesize $\disc X_{8} = %
      2^{-27} \cdot 3^{-27} \cdot 7^{9} \cdot 19^{18} \cdot I_{27}^{\min}$}
    where
    {\footnotesize
      \begin{math}
        I_{27}^{\min} = %
        2^{43}\cdot3^{27}\cdot7^{15}\cdot499^{14}\,.
      \end{math}
    }
  \end{dsuspend}
  \begin{dsuspend}
  \end{dsuspend}
  \begin{dmath*}
    {\Xb_{9}}:    96128\,{{x}}^{4}+232804\,{{x}}^{3}{y}+5588\,{{x}}^{3}{z}+51333\,{{x}}^{2}{{y}}^{2}-37020\,{{x}}^{2}{y}\,{z}-5791396\,{{x}}^{2}{{z}}^{2}-108416\,{x}\,{{y}}^{3}-49056\,{x}\,{{y}}^{2}{z}-6947226\,{x}\,{y}\,{{z}}^{2}-214292\,{x}\,{{z}}^{3}-5880\,{{y}}^{4}-581812\,{{y}}^{3}{z}+2438436\,{{y}}^{2}{{z}}^{2}+1944852\,{y}\,{{z}}^{3}+87102093\,{{z}}^{4}{=0}
  \end{dmath*}
  \begin{dsuspend}
    with {\footnotesize $\disc X_{9} = %
      2^{-27} \cdot 3^{-18} \cdot 13^{18} \cdot I_{27}^{\min}$}
    where
    {\footnotesize
      \begin{math}
        I_{27}^{\min} = %
        -2^{42}\cdot3^{18}\cdot5^{12}\cdot7^{14}\cdot79^{14}\cdot233^{14}\cdot857^{14}\,.
      \end{math}
    }
  \end{dsuspend}
  \begin{dsuspend}
  \end{dsuspend}
  \begin{dmath*}
    {\Xb_{10}}:    348\,{{x}}^{4}-832\,{{x}}^{3}{y}-4\,{{x}}^{3}{z}+261\,{{x}}^{2}{{y}}^{2}-132\,{{x}}^{2}{y}\,{z}-1680\,{{x}}^{2}{{z}}^{2}+224\,{x}\,{{y}}^{3}-168\,{x}\,{{y}}^{2}{z}+1986\,{x}\,{y}\,{{z}}^{2}+36\,{x}\,{{z}}^{3}+8\,{{y}}^{4}-236\,{{y}}^{3}{z}+404\,{{y}}^{2}{{z}}^{2}+428\,{y}\,{{z}}^{3}+1989\,{{z}}^{4}{=0}
  \end{dmath*}
  \begin{dsuspend}
    with {\footnotesize $\disc X_{10} = %
      2^{-27} \cdot 3^{-18} \cdot I_{27}^{\min}$}
    where
    {\footnotesize
      \begin{math}
        I_{27}^{\min} =  %
        -2^{42}\cdot3^{18}\cdot7^{14}\cdot41^{14}\cdot71^{14}\,.
      \end{math}
    }
  \end{dsuspend}
  \begin{dsuspend}
  \end{dsuspend}
  \begin{dmath*}
    {\Xb_{11}}:    245137\,{{x}}^{4}+3134444\,{{x}}^{3}{y}-405198\,{{x}}^{3}{z}+13885332\,{{x}}^{2}{{y}}^{2}-4713906\,{{x}}^{2}{y}\,{z}-6576142\,{{x}}^{2}{{z}}^{2}+25220768\,{x}\,{{y}}^{3}-13466052\,{x}\,{{y}}^{2}{z}-40450004\,{x}\,{y}\,{{z}}^{2}+6168379\,{x}\,{{z}}^{3}+16002624\,{{y}}^{4}-12848080\,{{y}}^{3}{z}-51202207\,{{y}}^{2}{{z}}^{2}+21339374\,{y}\,{{z}}^{3}+44888767\,{{z}}^{4}{=0}
  \end{dmath*}
  \begin{dsuspend}
    with {\footnotesize $\disc X_{11} = %
      2^{-27} \cdot 3^{-18} \cdot 31^{18} \cdot I_{27}^{\min}$}
    where
    {\footnotesize
      \begin{math}
        I_{27}^{\min} = %
        -2^{72}\cdot3^{18}\cdot7^{14}\cdot23^{14}\cdot47^{14}\cdot27527^{14}\,.
      \end{math}
    }
  \end{dsuspend}
  \begin{dsuspend}
  \end{dsuspend}
  \begin{dmath*}
    {\Xb_{12}}:    -2283766\,{{x}}^{4}-40282205\,{{x}}^{3}{y}+65256060\,{{x}}^{3}{z}+86351004\,{{x}}^{2}{{y}}^{2}-44980176\,{{x}}^{2}{y}\,{z}-98227040\,{{x}}^{2}{{z}}^{2}+34948793\,{x}\,{{y}}^{3}+112406040\,{x}\,{{y}}^{2}{z}-10691928\,{x}\,{y}\,{{z}}^{2}-811765633\,{x}\,{{z}}^{3}-46977843\,{{y}}^{4}+27242836\,{{y}}^{3}{z}+210065028\,{{y}}^{2}{{z}}^{2}-159829005\,{y}\,{{z}}^{3}-57425706\,{{z}}^{4}{=0}
  \end{dmath*}
  \begin{dsuspend}
    with {\footnotesize $\disc X_{12} = %
      -2^{-45} \cdot 3^{-18} \cdot I_{27}^{\min}$}
    where
    {\footnotesize
      \begin{math}
        I_{27}^{\min} = %
        2^{5}\cdot3^{18}\cdot7^{14}\cdot11^{9}\cdot5711^{14}\cdot73064203493^{14}\,.
      \end{math}
    }
  \end{dsuspend}
  \begin{dsuspend}
  \end{dsuspend}
  \begin{dmath*}
    {\Xb_{13}}:    13741849\,{{x}}^{4}-33952358\,{{x}}^{3}{y}-12314654\,{{x}}^{3}{z}-79058925\,{{x}}^{2}{{y}}^{2}+321820356\,{{x}}^{2}{y}\,{z}-449435767\,{{x}}^{2}{{z}}^{2}+24161786\,{x}\,{{y}}^{3}+58585032\,{x}\,{{y}}^{2}{z}+184173924\,{x}\,{y}\,{{z}}^{2}+202615424\,{x}\,{{z}}^{3}+10642401\,{{y}}^{4}+150598482\,{{y}}^{3}{z}+136602159\,{{y}}^{2}{{z}}^{2}-6607170137\,{y}\,{{z}}^{3}+3720024064\,{{z}}^{4}{=0}
  \end{dmath*}
  \begin{dsuspend}
    with {\footnotesize $\disc X_{13} = %
      2^{-36} \cdot 3^{-18} \cdot 11^{18} \cdot  43^{18} \cdot I_{27}^{\min}$}
    where
    {\footnotesize
      \begin{math}
        I_{27}^{\min} = %
        2^{20}\cdot3^{18}\cdot11^{9}\cdot547^{14}\cdot11827^{14}\cdot189169^{14}\,.
      \end{math}
    }
  \end{dsuspend}
  \begin{dsuspend}
  \end{dsuspend}
  \begin{dmath*}
    {\Xb_{14}}:    727950\,{{x}}^{4}-1982567\,{{x}}^{3}{y}-1449460\,{{x}}^{3}{z}+2619975\,{{x}}^{2}{{y}}^{2}-7272852\,{{x}}^{2}{y}\,{z}+12943560\,{{x}}^{2}{{z}}^{2}+1222070\,{x}\,{{y}}^{3}-9541020\,{x}\,{{y}}^{2}{z}-10154664\,{x}\,{y}\,{{z}}^{2}+31717821\,{x}\,{{z}}^{3}+3907465\,{{y}}^{4}+7463256\,{{y}}^{3}{z}+4691252\,{{y}}^{2}{{z}}^{2}+58884154\,{y}\,{{z}}^{3}+10671882\,{{z}}^{4}{=0}
  \end{dmath*}
  \begin{dsuspend}
    with {\footnotesize $\disc X_{14} = %
      2^{-45} \cdot 3^{-18} \cdot 19^{18} \cdot  I_{27}^{\min}$}
    where
    {\footnotesize
      \begin{math}
        I_{27}^{\min} = %
        2^{5}\cdot3^{18}\cdot11^{19}\cdot101^{14}\cdot107^{14}\cdot8378707^{14}\,.
      \end{math}
    }
  \end{dsuspend}
  \begin{dsuspend}
  \end{dsuspend}
  \begin{dmath*}
    {\Xb_{15}}:    {{x}}^{4}-{{x}}^{3}{y}+2\,{{x}}^{3}{z}+2\,{{x}}^{2}{y}\,{z}+2\,{{x}}^{2}{{z}}^{2}-2\,{x}\,{{y}}^{2}{z}+4\,{x}\,{y}\,{{z}}^{2}-{{y}}^{3}{z}+3\,{{y}}^{2}{{z}}^{2}+2\,{y}\,{{z}}^{3}+{{z}}^{4}{=0}
  \end{dmath*}
  \begin{dsuspend}
    with {\footnotesize $\disc X_{15} = %
      2^{-45} \cdot 3^{-27} \cdot I_{27}^{\min}$}
    where
    {\footnotesize
      \begin{math}
        I_{27}^{\min} = %
        2^{5}\cdot3^{27}\cdot19^{7}\,.
      \end{math}
    }
  \end{dsuspend}
  \begin{dsuspend}
  \end{dsuspend}
  \begin{dmath*}
    {\Xb_{16}}:    66648606\,{{x}}^{4}-10422787\,{{x}}^{3}{y}-1171743077\,{{x}}^{3}{z}+272093232\,{{x}}^{2}{{y}}^{2}+894539212\,{{x}}^{2}{y}\,{z}+1758438152\,{{x}}^{2}{{z}}^{2}-239684773\,{x}\,{{y}}^{3}-3355325973\,{x}\,{{y}}^{2}{z}+21854285561\,{x}\,{y}\,{{z}}^{2}+213880974126\,{x}\,{{z}}^{3}+731104019\,{{y}}^{4}-6282157788\,{{y}}^{3}{z}-38790710054\,{{y}}^{2}{{z}}^{2}+288506848419\,{y}\,{{z}}^{3}+1153356733618\,{{z}}^{4}{=0}
  \end{dmath*}
  \begin{dsuspend}
    with {\footnotesize $\disc X_{16} = %
      2^{-45} \cdot 3^{-27} \cdot 19^{18} \cdot I_{27}^{\min}$}
    where
    {\footnotesize
      \begin{math}
        I_{27}^{\min} = %
        2^{5}\cdot3^{35}\cdot19^{7}\cdot37^{14}\cdot79^{14}\cdot13373064392147^{14}\,.
      \end{math}
    }
  \end{dsuspend}
  \begin{dsuspend}
  \end{dsuspend}
  \begin{dmath*}
    {\Xb_{17}}:    3717829\,{{x}}^{4}-1434896\,{{x}}^{3}{y}+19525079\,{{x}}^{3}{z}-23623031\,{{x}}^{2}{{y}}^{2}+55253545\,{{x}}^{2}{y}\,{z}+168545160\,{{x}}^{2}{{z}}^{2}+36024736\,{x}\,{{y}}^{3}-64558785\,{x}\,{{y}}^{2}{z}+379342822\,{x}\,{y}\,{{z}}^{2}-329255097\,{x}\,{{z}}^{3}+42096963\,{{y}}^{4}+115245505\,{{y}}^{3}{z}-817353798\,{{y}}^{2}{{z}}^{2}+498157725\,{y}\,{{z}}^{3}-34967215\,{{z}}^{4}{=0}
  \end{dmath*}
  \begin{dsuspend}
    with {\footnotesize $\disc X_{17} = %
      2^{-36} \cdot 3^{-27} \cdot I_{27}^{\min}$}
    where
    {\footnotesize
      \begin{math}
        I_{27}^{\min} = %
        2^{20}\cdot3^{77}\cdot19^{7}\cdot1229^{14}\cdot3913841117^{14}\,.
      \end{math}
    }
  \end{dsuspend}
  \begin{dsuspend}
  \end{dsuspend}
  \begin{dmath*}
    {\Xb_{18}}:    3278898472\,{{x}}^{4}+35774613556\,{{x}}^{3}{y}-172165788624\,{{x}}^{3}{z}-42633841878\,{{x}}^{2}{{y}}^{2}+224611458828\,{{x}}^{2}{y}\,{z}+362086824567\,{{x}}^{2}{{z}}^{2}+6739276447\,{x}\,{{y}}^{3}+195387780024\,{x}\,{{y}}^{2}{z}+1153791743988\,{x}\,{y}\,{{z}}^{2}-3461357269578\,{x}\,{{z}}^{3}-18110161476\,{{y}}^{4}-549025255626\,{{y}}^{3}{z}-482663555556\,{{y}}^{2}{{z}}^{2}+15534718882176\,{y}\,{{z}}^{3}-61875497274721\,{{z}}^{4}{=0}
  \end{dmath*}
  \begin{dsuspend}
    with {\footnotesize $\disc X_{18} = %
      2^{-36} \cdot 13^{18} \cdot I_{27}^{\min}$}
    where
    {\footnotesize
      \begin{math}
        I_{27}^{\min} = %
        -2^{32}\cdot19^{7}\cdot101^{14}\cdot251^{14}\cdot7468843725186901^{14}\,.
      \end{math}
    }
  \end{dsuspend}
  \begin{dsuspend}
  \end{dsuspend}
  \begin{dmath*}
    {\Xb_{19}}:    -7\,{{x}}^{4}-2\,{{x}}^{3}{y}+16\,{{x}}^{3}{z}+7\,{{x}}^{2}{{y}}^{2}-6\,{{x}}^{2}{y}\,{z}-8\,{{x}}^{2}{{z}}^{2}+10\,{x}\,{{y}}^{3}+14\,{x}\,{{y}}^{2}{z}+2\,{x}\,{y}\,{{z}}^{2}-15\,{x}\,{{z}}^{3}+{{y}}^{4}+10\,{{y}}^{3}{z}+13\,{{y}}^{2}{{z}}^{2}+17\,{y}\,{{z}}^{3}+14\,{{z}}^{4}{=0}
  \end{dmath*}
  \begin{dsuspend}
    with {\footnotesize $\disc X_{19} = %
      2^{-36} \cdot 3^{-27} \cdot I_{27}^{\min}$}
    where
    {\footnotesize
      \begin{math}
        I_{27}^{\min} = %
        2^{20}\cdot3^{27}\cdot11^{14}\cdot43^{7}\,.
      \end{math}
    }
  \end{dsuspend}
  \begin{dsuspend}
  \end{dsuspend}
  \begin{dmath*}
    {\Xb_{20}}: 42978499\,{{x}}^{4}+91609890\,{{x}}^{3}{y}+226411413\,{{x}}^{3}{z}-152950386\,{{x}}^{2}{{y}}^{2}+225973290\,{{x}}^{2}{y}\,{z}+64073952\,{{x}}^{2}{{z}}^{2}+26287800\,{x}\,{{y}}^{3}+11918208\,{x}\,{{y}}^{2}{z}-742181730\,{x}\,{y}\,{{z}}^{2}-464894250\,{x}\,{{z}}^{3}-29463649\,{{y}}^{4}+198058830\,{{y}}^{3}{z}-144994689\,{{y}}^{2}{{z}}^{2}-208213515\,{y}\,{{z}}^{3}+85424183\,{{z}}^{4}{=0}
  \end{dmath*}
  \begin{dsuspend}
    with {\footnotesize $\disc X_{20} = %
      -2^{-45} \cdot I_{27}^{\min}$}
    where
    {\footnotesize
      \begin{math}
        I_{27}^{\min} = %
        2^{5}\cdot67^{7}\cdot1439^{14}\cdot2739021126001^{14}\,.
      \end{math}
    }
  \end{dsuspend}
\end{dgroup*}


\begin{thebibliography}{10}

\bibitem{harris}
E.~Arbarello, M.~Cornalba, P.~A. Griffiths, and J.~Harris.
\newblock {\em Geometry of algebraic curves. {V}ol. {I}}, volume 267 of {\em
  Grundlehren der Mathematischen Wissenschaften}.
\newblock Springer-Verlag, New York, 1985.

\bibitem{BILVcode}
J.~S. Balakrishnan, S.~Ionica, K.~Lauter, and C.~Vincent.
\newblock Genus 3.
\newblock Package available at
  \href{https://github.com/christellevincent/genus3}{\tt
  https://github.com/christellevincent/genus3}.

\bibitem{BILV}
J.~S. {Balakrishnan}, S.~{Ionica}, K.~{Lauter}, and C.~{Vincent}.
\newblock {Constructing genus 3 hyperelliptic Jacobians with CM}.
\newblock \href{http://arxiv.org/abs/1603.03832}{\tt arXiv:1603.03832}, 2016.

\bibitem{pari-gp}
C.~Batut, K.~Belabas, D.~Benardi, H.~Cohen, and M.~Olivier.
\newblock {\em User's Guide to {PARI-GP}}, 1998.
\newblock \href{https://pari.math.u-bordeaux.fr}{\tt
  https://pari.math.u-bordeaux.fr}.

\bibitem{Magma}
W.~Bosma, J.~Cannon, and C.~Playoust.
\newblock The {M}agma algebra system. {I}. {T}he user language.
\newblock {\em J. Symbolic Comput.}, 24(3-4):235--265, 1997.
\newblock Computational algebra and number theory (London, 1993).

\bibitem{BCLLMNO}
I.~Bouw, J.~Cooley, K.~Lauter, E.~Lorenzo Garc\'\i~a, M.~Manes, R.~Newton, and
  E.~Ozman.
\newblock Bad reduction of genus three curves with complex multiplication.
\newblock In {\em Women in numbers {E}urope}, volume~2 of {\em Assoc. Women
  Math. Ser.}, pages 109--151. Springer, Cham, 2015.

\bibitem{BouyerStreng}
F.~Bouyer and M.~Streng.
\newblock Examples of {CM} curves of genus two defined over the reflex field.
\newblock {\em LMS J. Comput. Math.}, 18(1):507--538, 2015.

\bibitem{broeker-stevenhagen-cm}
R.~Br{\"o}ker and P.~Stevenhagen.
\newblock Elliptic curves with a given number of points.
\newblock In {\em Algorithmic number theory}, volume 3076 of {\em Lecture Notes
  in Comput. Sci.}, pages 117--131. Springer, Berlin, 2004.

\bibitem{cmsv-article}
E.~Costa, N.~Mascot, J.~Sijsling, and J.~Voight.
\newblock Rigorous computation of the endomorphism ring of a {J}acobian.
\newblock \href{http://arxiv.org/abs/1705.09248}{\tt arXiv:1707.01158}, 2016.

\bibitem{DHBHS}
B.~Deconinck, M.~Heil, A.~Bobenko, M.~van Hoeij, and M.~Schmies.
\newblock Computing {R}iemann theta functions.
\newblock {\em Math. Comp.}, 73(247):1417--1442, 2004.

\bibitem{Deconinck}
B.~Deconinck and M.~van Hoeij.
\newblock Computing {R}iemann matrices of algebraic curves.
\newblock {\em Phys. D}, 152/153:28--46, 2001.
\newblock Advances in nonlinear mathematics and science.

\bibitem{dixmier}
J.~Dixmier.
\newblock On the projective invariants of quartic plane curves.
\newblock {\em Adv. in Math.}, 64:279--304, 1987.

\bibitem{Dupont}
R.~Dupont.
\newblock {\em Moyenne arithm{\'e}tico-g{\'e}om{\'e}trique, suites de Borchardt
  et applications}.
\newblock PhD thesis, {\'E}cole polytechnique, Palaiseau, 2006.

\bibitem{elsenhans-good}
A.-S. Elsenhans.
\newblock Good models for cubic surfaces.
\newblock Preprint at
  \href{https://math.uni-paderborn.de/fileadmin/mathematik/AG-Computeralgebra/Preprints-elsenhans/red_5.pdf}{\tt
  https://math.uni-pader\-born.de/\-fileadmin/\-mathematik/\-AG-Com\-puter\-alge\-bra/\-Pre\-prints-elsen\-hans/\-red\_5.pdf}.

\bibitem{elsenhans}
A.-S. Elsenhans.
\newblock Explicit computations of invariants of plane quartic curves.
\newblock {\em J. Symbolic Comput.}, 68(part 2):109--115, 2015.

\bibitem{fiorentino}
A.~Fiorentino.
\newblock Weber's formula for the bitangents of a smooth plane quartic.
\newblock \href{http://arxiv.org/abs/1612.02049}{\tt arXiv:1612.02049}, 2016.

\bibitem{giko}
M.~Girard and D.~R. Kohel.
\newblock {Classification of genus 3 curves in special strata of the moduli
  space.}
\newblock {In Hess, F. (ed.) et al., Algorithmic number theory. 7th
  international symposium, ANTS-VII, Berlin, Germany, July 23--28, 2006.
  Proceedings. Berlin: Springer. Lecture Notes in Computer Science 4076,
  346-360.}, 2006.

\bibitem{GorenLauter}
E.~Z. Goren and K.~E. Lauter.
\newblock Class invariants for quartic {CM} fields.
\newblock {\em Ann. Inst. Fourier (Grenoble)}, 57(2):457--480, 2007.

\bibitem{grossharris}
B.~H. Gross and J.~Harris.
\newblock On some geometric constructions related to theta characteristics.
\newblock In {\em Contributions to automorphic forms, geometry, and number
  theory}, pages 279--311. Johns Hopkins Univ. Press, Baltimore, MD, 2004.

\bibitem{guardia}
J.~Gu{\`a}rdia.
\newblock On the {T}orelli problem and {J}acobian {N}ullwerte in genus three.
\newblock \href{http://arxiv.org/abs/0901.4376}{\tt arXiv:0901.4376}, 2009.

\bibitem{igusa1}
J.-I. Igusa.
\newblock {\em Theta functions}, volume 194 of {\em Grundlehren der
  mathematischen Wissenschaften}.
\newblock Springer-Verlag, New York, 1972.

\bibitem{kilicerthesis}
P.~K{\i}l{\i}\c{c}er.
\newblock {\em The {CM} class number one problem for curves}.
\newblock PhD thesis, Universiteit Leiden and Universit\'e de Bordeaux, 2016.

\bibitem{fullistofcurves}
P.~{K\i{}l\i{}\c{c}er}, H.~Labrande, R.~Lercier, C.~Ritzenthaler, J.~Sijsling,
  and M.~Streng.
\newblock {CM} plane quartics.
\newblock Available at
  \href{https://arxiv.org/src/1701.06489v3/anc/curves.txt}{\tt
  https://arxiv.org/src/1701.06489v3/anc/curves.txt}.

\bibitem{KLLNOS}
P.~{K\i{}l\i{}\c{c}er}, K.~Lauter, E.~Lorenzo~Garc\'ia, R.~Newton, E.~Ozman,
  and M.~Streng.
\newblock A bound on the primes of bad reduction for {CM} curves of genus 3.
\newblock \href{https://arxiv.org/abs/1609.05826}{\tt arXiv:1609.05826}, 2016.

\bibitem{KLS}
P.~{K\i{}l\i{}\c{c}er}, E.~Lorenzo~Garc\'ia, and M.~Streng.
\newblock Primes dividing invariants of {CM} {P}icard curves.
\newblock \href{http://arxiv.org/abs/1801.04682}{\tt arXiv:1801.04682}, 2018.

\bibitem{KS2016}
P.~K{\i}l{\i}\c{c}er and M.~Streng.
\newblock Rational {CM} points and class polynomials for genus three.
\newblock In preparation, 2016.

\bibitem{kowe}
K.~Koike and A.~Weng.
\newblock Construction of {CM} {P}icard curves.
\newblock {\em Math. Comp.}, 74(249):499--518 (electronic), 2005.

\bibitem{koizumi}
S.~Koizumi.
\newblock The fields of moduli for polarized abelian varieties and for curves.
\newblock {\em Nagoya Math. J.}, 48:37--55, 1972.

\bibitem{ourimplementation}
H.~Labrande.
\newblock Calculating theta functions in genus {$3$}.
\newblock Package available at
  \href{http://hlabrande.fr/pubs/fastthetaconstantsgenus3.m}{\tt
  http://hlabrande.fr/pubs/fast\-thetaconstantsgenus3.m}.

\bibitem{Labrande}
H.~Labrande.
\newblock {\em {Explicit computation of the Abel-Jacobi map and its inverse}}.
\newblock PhD thesis, {Universit{\'e} de Lorraine and University of Calgary},
  Nov. 2016.

\bibitem{labrande-thome}
H.~Labrande and E.~Thom\'e.
\newblock Computing theta functions in quasi-linear time in genus two and
  above.
\newblock {\em LMS J. Comput. Math.}, 19(suppl. A):163--177, 2016.

\bibitem{LangCM}
S.~Lang.
\newblock {\em Complex multiplication}, volume 255 of {\em Grundlehren der
  Mathematischen Wissenschaften}.
\newblock Springer-Verlag, New York, 1983.

\bibitem{LarioSomoza}
J.-C. Lario and A.~Somoza.
\newblock A note on {P}icard curves of {CM}-type.
\newblock \href{http://arxiv.org/abs/1611.02582}{\tt arXiv:1611.02582}, 2016.

\bibitem{lauter}
K.~Lauter.
\newblock Geometric methods for improving the upper bounds on the number of
  rational points on algebraic curves over finite fields.
\newblock {\em J. Algebraic Geom.}, 10(1):19--36, 2001.
\newblock With an appendix by J.-P. Serre.

\bibitem{viray}
K.~Lauter and B.~Viray.
\newblock An arithmetic intersection formula for denominators of {I}gusa class
  polynomials.
\newblock {\em Amer. J. Math.}, 137(2):497--533, 2015.

\bibitem{LR08}
R.~Lercier and C.~Ritzenthaler.
\newblock Invariants and reconstructions for genus $2$ curves in any
  characteristic, 2008.
\newblock Available in \textsc{Magma} 2.15 and later.

\bibitem{LR11}
R.~Lercier and C.~Ritzenthaler.
\newblock Hyperelliptic curves and their invariants: geometric, arithmetic and
  algorithmic aspects.
\newblock {\em J. Algebra}, 372:595--636, 2012.

\bibitem{LRRS14}
R.~Lercier, C.~Ritzenthaler, F.~Rovetta, and J.~Sijsling.
\newblock Parametrizing the moduli space of curves and applications to smooth
  plane quartics over finite fields.
\newblock {\em LMS J.\ Comput. Math.}, 17(suppl. A):128--147, 2014.

\bibitem{lrs}
R.~Lercier, C.~Ritzenthaler, and J.~Sijsling.
\newblock Fast computation of isomorphisms of hyperelliptic curves and explicit
  descent.
\newblock In E.~W. Howe and K.~S. Kedlaya, editors, {\em Proceedings of the
  Tenth Algorithmic Number Theory Symposium}, pages 463--486. Mathematical
  Sciences Publishers, 2012.

\bibitem{LRS-Desc}
R.~Lercier, C.~Ritzenthaler, and J.~Sijsling.
\newblock Explicit {G}alois obstruction and descent for hyperelliptic curves
  with tamely cyclic reduced automorphism group.
\newblock {\em Math. Comp.}, 85(300):2011--2045, 2016.

\bibitem{LRS16-Code}
R.~Lercier, C.~Ritzenthaler, and J.~Sijsling.
\newblock quartic\_reconstruction; a \textsc{Magma} package for reconstructing
  plane quartics from {D}ixmier-{O}hno invariants.
\newblock Package available at
  \href{https://github.com/JRSijsling/quartic\_reconstruction}{\tt
  https://github.com/\-JRSijsling/\-quar\-tic\_re\-con\-struc\-tion}, 2016.

\bibitem{LRS16}
R.~Lercier, C.~Ritzenthaler, and J.~Sijsling.
\newblock Reconstructing plane quartics from their invariants.
\newblock \href{http://arxiv.org/abs/1606.05594}{\tt arXiv:1606.05594}, 2016.

\bibitem{LLR17}
E.~Lorenzo~Garc\'ia, R.~Lercier, and C.~Ritzenthaler.
\newblock Reduction type of non-hyperelliptic genus 3 curves.
\newblock In preparation, 2017.

\bibitem{mestre}
J.-F. Mestre.
\newblock Construction de courbes de genre $2$ \`a partir de leurs modules.
\newblock In {\em Effective methods in algebraic geometry}, volume~94 of {\em
  Prog. Math.}, pages 313--334, Boston, 1991. Birk{\"a}user.

\bibitem{rauch}
H.~E. Rauch and H.~M. Farkas.
\newblock {\em Theta functions with applications to {R}iemann surfaces}.
\newblock The Williams\thinspace \&\thinspace Wilkins Co., Baltimore, Md.,
  1974.

\bibitem{agmri}
C.~Ritzenthaler.
\newblock Point counting on genus 3 non hyperelliptic curves.
\newblock In {\em Algorithmic number theory}, volume 3076 of {\em Lecture Notes
  in Comput. Sci.}, pages 379--394. Springer, Berlin, 2004.

\bibitem{SerreTate}
J.-P. Serre and J.~Tate.
\newblock Good reduction of abelian varieties.
\newblock {\em Ann. of Math. (2)}, 88:492--517, 1968.

\bibitem{shimura-moduli}
G.~Shimura.
\newblock On the field of rationality for an abelian variety.
\newblock {\em Nagoya Math. J.}, 45:167--178, 1971.

\bibitem{Shimura77}
G.~Shimura.
\newblock On abelian varieties with complex multiplication.
\newblock {\em Proc. London Math. Soc. (3)}, 34(1):65--86, 1977.

\bibitem{shimura61:_compl_abelian}
G.~Shimura and Y.~Taniyama.
\newblock Complex multiplication of {A}belian varieties and its applications to
  number theory.
\newblock {\em Publ. Math. Soc. Japan}, 6, 1961.

\bibitem{simon}
D.~Simon.
\newblock Solving quadratic equations using reduced unimodular quadratic forms.
\newblock Preprint available at
  \href{http://simond.users.lmno.cnrs.fr/qfsolve.gp}{\tt
  http://simond.users.lmno.cnrs.fr/\-qfsolve.gp}, 2013.

\bibitem{spallek}
A.-M. Spallek.
\newblock {\em Kurven vom {G}eschlecht 2 und ihre {A}nwendung in
  {P}ublik-{K}ey-{K}ryptosystemen}.
\newblock PhD thesis, Institut f{\"u}r {E}xperimentelle {M}athematik, Essen,
  1994.

\bibitem{Stoll2011}
M.~Stoll.
\newblock Reduction theory of point clusters in projective space.
\newblock {\em Groups Geom. Dyn.}, 5(2):553--565, 2011.

\bibitem{strengrecip}
M.~Streng.
\newblock Recip -- Repository of complex multiplication SageMath code (formerly
  Sage package for using {S}himura's reciprocity law for {S}iegel modular
  functions).
\newblock \href{http://www.math.leidenuniv.nl/~streng/recip}{\tt
  http://www.math.leidenuniv.nl/\-{$\sim$}streng/\-recip}.

\bibitem{streng}
M.~Streng.
\newblock Computing {I}gusa class polynomials.
\newblock {\em Math. Comp.}, 83(285):275--309, 2014.
\newblock \href{http://arxiv.org/abs/0903.4766}{\tt arXiv:0903.4766}.

\bibitem{tate2}
J.~Tate.
\newblock Classes d'isogénie des variétés abéliennes sur un corps fini
  (d'après {H}onda).
\newblock In {\em S{\'e}minaire {B}ourbaki 1968/69}, volume 179 of {\em Lecture
  Notes in Math.}, pages 95--110. Springer, Berlin, 1971.

\bibitem{TT91}
W.~Tautz, J.~Top, and A.~Verberkmoes.
\newblock {Explicit hyperelliptic curves with real multiplication and
  permutation polynomials}.
\newblock {\em Canad. J. Math.}, 43(5):1055--1064, 1991.

\bibitem{sage}
{The Sage Developers}.
\newblock {\em {S}ageMath, the {S}age {M}athematics {S}oftware {S}ystem
  ({V}ersion 7.4)}, 2016.
\newblock {\tt http://www.sagemath.org}.

\bibitem{VanRijnswou}
S.~M. van Rijnswou.
\newblock {\em Testing the equivalence of planar curves}.
\newblock PhD thesis, Technische Universiteit Eindhoven, Eindhoven, 2001.

\bibitem{wamelen}
P.~van Wamelen.
\newblock Examples of genus two {CM} curves defined over the rationals.
\newblock {\em Math. Comp.}, 68(225):307--320, 1999.

\bibitem{voight}
J.~Voight.
\newblock Identifying the matrix ring: algorithms for quaternion algebras and
  quadratic forms.
\newblock In {\em Quadratic and higher degree forms}, volume~31 of {\em Dev.
  Math.}, pages 255--298. Springer, New York, 2013.

\bibitem{weber}
H.~Weber.
\newblock {Theory of abelian functions of genus 3. (Theorie der Abel'schen
  Functionen vom Geschlecht 3.)}, 1876.

\bibitem{wenghypg3}
A.~Weng.
\newblock A class of hyperelliptic {CM}-curves of genus three.
\newblock {\em J. Ramanujan Math. Soc.}, 16(4):339--372, 2001.

\end{thebibliography}
\end{document}